\documentclass[11pt]{article}

\usepackage[utf8]{inputenc}
\usepackage[english]{babel}
\usepackage{amsmath}
\usepackage{nicefrac}
\usepackage{amsthm}
\usepackage{amsfonts}
\usepackage{verbatim}
\usepackage{color}
\usepackage[arrow, matrix, curve]{xy}
\usepackage{bbm}
\usepackage[numbers]{natbib}
\usepackage{amssymb}
\usepackage{xcolor}

\DeclareMathSymbol{\leq}{\mathrel}{symbols}{20}
   \let\le=\leq
\DeclareMathSymbol{\geq}{\mathrel}{symbols}{21}
   \let\ge=\geq

\newcommand{\R}{\mathbb{R}}

\newcommand{\N}{\mathbb{N}}

\newcommand{\F}{\mathcal{F}}

\renewcommand{\P}{\mathbb{P}}
\newcommand{\E}{\mathbb{E}}
\newcommand{\e}{\varepsilon}

\renewcommand{\1}{\mathbbm{1}}

\newtheorem{Theorem}{Theorem}[section]
\newtheorem{Proposition}[Theorem]{Proposition}
\newtheorem{Corollary}[Theorem]{Corollary}
\newtheorem{Lemma}[Theorem]{Lemma}
\newtheorem{Remark}[Theorem]{Remark}
\newtheorem{Definition}[Theorem]{Definition}

\numberwithin{equation}{section}

\makeatletter
\newcommand{\customlabel}[1]{%
     \stepcounter{ref}%
   \protected@write
\@auxout{}{\string\newlabel{#1}{{\thesatz.\arabic{ref}}{\thepage}{\thesatz.\arabic{ref}}{#1}{}}}%
   \hypertarget{#1}{\thesatz.\arabic{ref}}%
}
\makeatother

\newenvironment{sciabstract}{\begin{quote}}{\end{quote}}

\newcounter{lastnote}

\title{Stochastic equation and exponential ergodicity in Wasserstein distances for affine processes}
\newcommand{\pdftitle}{Stochastic equation and exponential ergodicity in Wasserstein distances for affine processes}
\newcommand{\pdfauthor}{Martin Friesen}

\author{
Martin Friesen\footnote{Fakult\"at f\"ur Mathematik und Naturwissenschaften, Bergische Universit\"at Wuppertal, Gaußstraße 20, 42119 Wuppertal, Germany, friesen@math.uni-wuppertal.de}\\
Peng Jin\footnote{Department of Mathematics, Shantou University, Shantou, Guangdong 515063, China, pjin@stu.edu.cn}\\
Barbara R\"udiger\footnote{Fakult\"at f\"ur Mathematik und Naturwissenschaften, Bergische Universit\"at Wuppertal, Gaußstraße 20, 42119 Wuppertal, Germany, ruediger@uni-wuppertal.de}
}

\usepackage[plainpages=false,pdfpagelabels=true,bookmarks=true,pdfauthor={\pdfauthor},
pdftitle={\pdftitle}]{hyperref}

\makeatletter
\def\HyPsd@CatcodeWarning#1{}
\makeatother

\begin{document}

\maketitle

\begin{sciabstract}\textbf{Abstract:}
This work is devoted to the study of conservative affine processes on the canonical state space $D = \R_+^m \times \R^n$, where $m+n > 0$.
We show that each affine process can be obtained as the pathwise unique strong solution to a stochastic equation
driven by Brownian motions and Poisson random measures.
Then we study the long-time behavior of affine processes, i.e., we show that under first moment condition on the state-dependent and
$\log$-moment conditions on the state-independent jump measures, respectively, each subcritical affine process is exponentially ergodic in a suitably chosen Wasserstein distance.
Moments of affine processes are studied as well.
\end{sciabstract}

\noindent \textbf{AMS Subject Classification:} 37A25; 60H10; 60J25\\
\textbf{Keywords:} affine process; ergodicity; Wasserstein distance; coupling; stochastic differential equation

\section{Introduction and statement of the result}

\subsection{General introduction}
An affine process is a time-homogeneous Markov processes $(X_t)_{t \geq 0}$ whose characteristic function satisfies
\[
 \E_x \left( e^{i \langle u, X_t \rangle} \right)
  = \exp\left ( \phi(t,iu)  +  \langle x, \psi(t,iu) \rangle \right),
\]
where $t \geq 0$ is the time and $X_0 = x$ the starting point of the process.
The general theory of affine processes, including a full characterization on the canonical state space $D = \R_+^m \times \R^n$
where $m,n \in \N_0$ and $m + n > 0$, was discussed in \cite{DFS03}.
In particular, it is shown that the functions $\phi$ and $\psi$ should satisfy certain generalized Riccati equations.
Common applications of affine processes in mathematical finance are
interest rate models (e.g., the Cox-Ingersoll-Ross, Vašiček or general affine term structure short rate models),
option pricing (e.g., the Heston model) and credit risk models,
see also \cite{A15} and the references therein.
After \cite{DFS03}, the mathematical theory of affine processes was developed in various directions.
Regularity of affine processes was studied in \cite{KST11} and \cite{KST13}.
Based on a H\"ormander-type condition, existence and smoothness
of transition densities were obtained in \cite{FMS13}.
Exponential moments for affine processes were studied in \cite{JKX12} and \cite{KM15}.
The theory of affine diffusions, i.e., processes without jumps,
was developed in \cite{FM09},
while its application to large deviations for affine diffusions
was studied in \cite{KK14}.
The possibility to obtain affine processes as multi-parameter time changes of L\'evy processes was recently discussed in \cite{CPU17}.
It is worthwhile to mention that the above list is, by far, not complete.
For further references and additional details on the general theory of affine processes we refer to the book \cite{A15}.

Below we describe two important sub-classes of affine processes.
\textit{Continuous-state branching processes with immigration}
(shorted as CBI processes) are affine processes
with state space $D = \R_+^m$.
Such processes have been first
introduced in 1958 by Ji\v{r}ina \cite{J58} and then
studied in \cite{W69,KW71,SW73},
where it was also shown that these processes arise as scaling limits of Galton-Watson processes.
Various properties of one-dimensional CBI processes were studied in
\cite{G74, FFS85, CPU13, KM12, FU14, DFM14} and \cite{CLP18}.
For results applicable in arbitrary dimension we refer to \cite{BLP15}, \cite{BLP16} and \cite{FJR18}. Let us mention that CBI processes are also measure-valued Markov processes as studied in \cite{L11}.
Another important class of affine processes corresponds to the
state space $D = \R^n$ and is consisted of processes of Ornstein-Uhlenbeck (OU) type.
These processes include also L\'evy processes as a particular case.

\subsection{Affine processes}

Let us describe affine processes in more detail. For $m,n\in\N_{0}$
let $d=n+m$, and suppose that $d>0$. In this work we study affine
processes on the canonical state space $D=\R_{+}^{m}\times\R^{n}$.
Let
\[
I=\{1,\dots,m\},\ \ J=\{m+1,\dots,d\}.
\]
If $x\in D$, then let $x_{I}=(x_{i})_{i\in I}$ and $x_{J}=(x_{j})_{j\in J}$.
Denote by $\R^{d\times d}$ the space of $d\times d$-matrices. For
$A\in\R^{d\times d}$ we write
\[
A=\left(\begin{array}{rrrr}
A_{II} & A_{IJ}\\
A_{JI} & A_{JJ}
\end{array}\right),
\]
where $A_{II}=(a_{ij})_{i,j\in I}$, $A_{IJ}=(a_{ij})_{i\in I,\ j\in J}$,
$A_{JI}=(a_{ij})_{i\in J,\ j\in I}$, and $A_{JJ}=(a_{ij})_{i,j\in J}$.
Denote by $S_{d}^{+}$ the space of symmetric and
positive semidefinite $d\times d$-matrices. Finally, let $\delta_{kl}$,
$k,l\in\{1,\dots,d\}$, stand for the Kronecker-Delta.
\begin{Definition}\label{defi: adm. parameter} We call a tuple $(a,\alpha,b,\beta,\nu,\mu)$
admissible parameters, if they satisfy the following conditions:
\begin{enumerate}
\item[(i)] $a\in S_{d}^{+}$ with $a_{II}=0$, $a_{IJ}=0$ and
$a_{JI}=0$.
\item[(ii)] $\alpha=(\alpha_{1},\dots,\alpha_{m})$ with $\alpha_{i}=(\alpha_{i,kl})_{1\leq k,l\leq d}\in S_{d}^{+}$
and $\alpha_{i,kl}=0$ if $k\in I\backslash\{i\}$
or $l\in I\backslash\{i\}$.
\item[(iii)] $b\in D$.
\item[(iv)] $\beta\in\R^{d\times d}$ is such that $\beta_{ki}-\int_{D} \xi_{k} \mu_{i}(d\xi)\geq0$
for all $i\in I$ and $k\in I\backslash\{i\}$, and $\beta_{IJ}=0$.
\item[(v)] $\nu$ is a Borel measure on $D$ such that $\nu(\{0\})=0$ and
\[
\int\limits _{D}\left(1\wedge|\xi|^{2}+\sum\limits _{i\in I}(1\wedge\xi_{i})\right)\nu(d\xi)<\infty.
\]
\item[(vi)] $\mu=(\mu_{1},\dots,\mu_{m})$ where $\mu_{1},\dots,\mu_{m}$ are
Borel measures on $D$ such that
\[
\mu_{i}(\{0\})=0,\qquad\int\limits _{D}\left(|\xi|\wedge|\xi|^{2}+\sum\limits _{k\in I\backslash\{i\}}\xi_{k}\right)\mu_{i}(d\xi)<\infty,\ \ i\in I.
\]
\end{enumerate}
\end{Definition} In contrast to \citep{DFS03}, we do not consider
killing for affine processes and, moreover, we suppose that $\mu_{1},\dots,\mu_{m}$
integrate $\1_{\{|\xi|>1\}}|\xi|$, i.e., the first
moment for big jumps is finite. It is well-known that without killing
and under first moment condition for the big jumps of $\mu_{1},\dots,\mu_{m}$,
the corresponding affine process (introduced below) is conservative
(see \citep[Lemma 9.2]{DFS03}). In this paper we
work with Definition \ref{defi: adm. parameter} and thus restrict
our study to conservative affine processes. In order to simplify
the notation, we have also set $\nu(\{0\})=0$ and $\mu_{i}(\{0\})=0$,
for $i\in I$. Hence all integrals with respect to the measures $\mu_{1},\dots,\mu_{m},\nu$
can be taken over $D$ instead of $D\backslash\{0\}$.

Denote by $B_{b}(D)$ the Banach space of bounded
measurable functions over $D$. This space is equipped with the supremum
norm $\|f\|_{\infty}=\sup_{x\in D}|f(x)|$. Define
\[
\mathcal{U}=\mathbb{C}_{\leq0}^{m}\times i\R^{n}=\{u=(u_{1},u_{2})\in\mathbb{C}^{m}\times\mathbb{C}^{n}\ |\ \mathrm{Re}(u_{1})\leq0,\ \ \mathrm{Re}(u_{2})=0\}.
\]
Note that $D\ni x\longmapsto e^{\langle u,x\rangle}$ is bounded for
any $u\in\mathcal{U}$. Here $\langle\cdot,\cdot\rangle$ denotes
the Euclidean scalar product on $\R^{d}$. By abuse
of notation, we later also use $\langle\cdot,\cdot\rangle$ for the
scalar product on $\R^{m}$ or $\R^{n}.$ The following is due to
\citep{DFS03}. \begin{Theorem}\label{EXISTENCE:AFFINE} Let $(a,\alpha,b,\beta,\nu,\mu)$
be admissible parameters. Then there exists a unique conservative
Feller semigroup $(P_{t})_{t\ge0}$ on $B_{b}(D)$
with generator $(L,D(L))$ such that $C_{c}^{2}(D)\subset D(L)$ and,
for $f\in C_{c}^{2}(D)$ and $x\in D$,
\begin{align*}
(Lf)(x) & =\langle b+\beta x,\nabla f(x)\rangle+\sum\limits _{k,l=1}^{d}\left(a_{kl}+\sum\limits _{i=1}^{m}\alpha_{i,kl}x_{i}\right)\frac{\partial^{2}f(x)}{\partial x_{k}\partial x_{l}}\\
 & \ \ \ +\int\limits _{D}\left(f(x+\xi)-f(x)-\langle\xi_{J},\nabla_{J}f(x)\rangle\1_{\{|\xi|\leq1\}}\right)\nu(d\xi)\\
 & \ \ \ +\sum\limits _{i=1}^{m}x_{i}\int\limits _{D}\left(f(x+\xi)-f(x)-\langle\xi,\nabla f(x)\rangle\right)\mu_{i}(d\xi),
\end{align*}
where $\nabla_{J}=(\frac{\partial}{\partial x_{j}})_{j\in J}$. Moreover,
$C_{c}^{\infty}(D)$ is a core for the generator. Let $P_{t}(x,dx')$
be the transition probabilities. Then
\begin{align}
\int\limits _{D}e^{\langle u,x'\rangle}P_{t}(x,dx')=\exp\left(\phi(t,u)+\langle x,\psi(t,u)\rangle\right),\qquad u\in\mathcal{U},\label{LAPLACE}
\end{align}
where $\phi:\R_{+}\times\mathcal{U}\longrightarrow\mathbb{C}$ and
$\psi:\R_{+}\times\mathcal{U}\longrightarrow\mathbb{C}^{d}$ are uniquely
determined by the generalized Riccati differential equations:
for $u=(u_{1},u_{2})\in\mathbb{C}_{\leq0}^{m}\times i\R^{n}$,
\begin{align}
\partial_{t}\phi(t,u) & =F(\psi(t,u)),\ \ \phi(0,u)=0,\label{RICATTI}\\
\partial_{t}\psi_{I}(t,u) & =R(\psi_{I}(t,u), e^{t \beta_{JJ}^{\top} }u_{2} ),\ \ \psi_{I}(0,u)=u_{1},\nonumber \\
\psi_{J}(t,u) & = e^{t \beta_{JJ}^{\top}}u_{2},\nonumber
\end{align}
and $F$, $R$ are of Lévy-Khintchine form
\begin{align*}
F(u) & =\langle u,au\rangle+\langle b,u\rangle+\int\limits _{D}\left(e^{\langle u,\xi\rangle}-1-\1_{\{|\xi|\leq1\}}\langle\xi_{J},u_{J}\rangle\right)\nu(d\xi),\\
R_{i}(u) & =\langle u,\alpha_{i}u\rangle+\sum\limits _{k=1}^{d}\beta_{ki}u_{k}+\int\limits _{D}\left(e^{\langle u,\xi\rangle}-1-\langle u,\xi\rangle\right)\mu_{i}(d\xi),\qquad i\in I.
\end{align*}
\end{Theorem} 
Consequently, there exists a unique Feller process
$X$ with generator $L$. This process is called affine process with
admissible parameters $(a,\alpha,b,\beta,\nu,\mu)$.

\begin{Remark}\label{COROLLARY:00} 
Let $(a,\alpha,b,\beta,\nu,\mu)$ be admissible parameters. According to \citep[Lemma 10.1 and Lemma 10.2]{DFS03},
the martingale problem with generator $L$ and domain $C_{c}^{\infty}(D)$
is well-posed in the Skorokhod space over $D$ equipped with the usual
Skorokhod topology. Hence, we can characterise an affine process with
admissible parameters $(a,\alpha,b,\beta,\nu,\mu)$ as the unique solution to the martingale problem with generator $L$ and domain $C_{c}^{\infty}(D)$. In any case, it can be constructed as a Markov process on the Skorokhod space over $D$. 
\end{Remark}
Affine processes are thus constructed on the canonical state space.
In order to prove the main result of this work, we provide in Section 4 a pathwise construction of affine processes. 
The latter one extends previous cases from the literature such as
\cite{DL06, FM09, M13} and \cite{BLP15}.

\subsection{Ergodicity in Wasserstein distance for affine processes}

Let $\mathcal{P}(D)$ be the space of all Borel probability measures
over $D$. By abuse of notation, we extend the transition semigroup
$(P_{t})_{t\ge0}$ (given by Theorem \ref{EXISTENCE:AFFINE}) onto
$\mathcal{P}(D)$ via
\begin{align}
(P_{t}\rho)(dx)=\int\limits _{D}P_{t}(\widetilde{x},dx)\rho(d\widetilde{x}),\qquad t\geq0,\ \ \rho\in\mathcal{P}(D).\label{EQ:05}
\end{align}
Then $P_{t}\rho$ describes the distribution of the affine process
at time $t\geq0$ such that it has at initial time $t=0$ law $\rho$.
Note that $P_{t}\delta_{x}=P_{t}(x,\cdot)$, and $(P_{t})_{t\ge0}$
is a semigroup on $\mathcal{P}(D)$ in the sense that $P_{t+s}\rho=P_{t}P_{s}\rho$,
for any $t,s\geq0$ and $\rho\in\mathcal{P}(D)$. Such semigroup property
is simply a compact notation for the Chapman-Kolmogorov equations
satisfied by $P_{t}(x,\cdot)$. Since the martingale problem with
generator $L$ and domain $C_{c}^{\infty}(D)$ is well-posed, and
$C_{c}^{\infty}(D)\subset D(L)$ is a core (see Theorem \ref{EXISTENCE:AFFINE}
and Remark \ref{COROLLARY:00}), it follows from \citep[Proposition 9.2]{EK86}
that, for some given $\pi\in\mathcal{P}(D)$, the following properties
are equivalent:
\begin{enumerate}
\item[(i)] $P_{t}\pi=\pi$, for all $t\geq0$.
\item[(ii)] $\int_{D}(Lf)(x)\pi(dx)=0$, for all $f\in C_{c}^{\infty}(D)$.
\item[(iii)] $\int_{D}(P_{t}f)(x)\pi(dx)=\int_{D}f(x)\pi(dx)$, for all $t\geq0$
and all $f\in B(D)$.
\end{enumerate}
A distribution $\pi\in\mathcal{P}(D)$ which satisfies one of these
properties (i) -- (iii) is called invariant distribution for the
semigroup $(P_{t})_{t\geq0}$.
In this work we will prove that, under some appropriate assumptions, $(P_t)_{t \geq 0}$
has a unique invariant distribution $\pi$,
this distribution has some finite $\log$-moment and,
moreover, $P_t(x,\cdot) \longrightarrow \pi$
with exponential rate.
For this purpose we use the Wasserstein distance
on $\mathcal{P}(D)$ introduced below. Given $\rho,\widetilde{\rho}\in\mathcal{P}(D)$,
a coupling $H$ of $(\rho,\widetilde{\rho})$ is a Borel probability
measure on $D\times D$ which has marginals $\rho$ and $\widetilde{\rho}$,
respectively, i.e., for $f,g\in B(D)$ it holds that
\[
\int\limits _{D\times D}\left(f(x)+g(\widetilde{x})\right)H(dx,d\widetilde{x})=\int\limits _{D}f(x)\rho(dx)+\int\limits _{D}g(x)\widetilde{\rho}(dx).
\]
Denote by $\mathcal{H}(\rho,\widetilde{\rho})$ the collection of
all such couplings. Let us now introduce two different metrics on
$D$ as follows:
\begin{enumerate}
\item[(a)] Define, for $\kappa\in(0,1]$, $d_{\kappa}(x,\widetilde{x})=\left(\1_{\{n>0\}}|y-\widetilde{y}|^{1/2}+|x-\widetilde{x}|\right)^{\kappa}$,
$x=(y,z),\ \widetilde{x}=(\widetilde{y},\widetilde{z})\in \R_{+}^{m}\times\R^{n}$,
and let
\[
\mathcal{P}_{d_{\kappa}}(D)=\left\{ \rho\in\mathcal{P}(D)\ |\ \int\limits _{D}|x|^{\kappa}\rho(dx)<\infty\right\} .
\]
\item[(b)] Introduce $d_{\log}(x,\widetilde{x})=\log(1+\1_{\{n>0\}}|y-\widetilde{y}|^{1/2}+|x-\widetilde{x}|)$,
$x=(y,z),\ \widetilde{x}=(\widetilde{y},\widetilde{z})\in \R_{+}^{m}\times\R^{n}$,
and let
\[
\mathcal{P}_{d_{\log}}(D)=\left\{ \rho\in\mathcal{P}(D)\ |\ \int\limits _{D}\log(1+|x|)\rho(dx)<\infty\right\} .
\]
\end{enumerate}
Let $d\in\{d_{\log},d_{\kappa}\}$. The Wasserstein distance on $\mathcal{P}_{d}(D)$
is defined by
\begin{align}\label{WASSERSTEIN:2}
W_{d}(\rho,\widetilde{\rho})=\inf\left\{ \int\limits _{D\times D}d(x,\widetilde{x})H(dx,d\widetilde{x})\ |\ H\in\mathcal{H}(\rho,\widetilde{\rho})\right\} .
\end{align}
The appearance of the additional factor $\1_{\{n>0\}}|y-\widetilde{y}|^{1/2}$
is purely technical, it is a consequence of the estimates proved in
Section 6. By general theory of Wasserstein distances we see that
$(\mathcal{P}_{d}(D),W_{d})$ is a complete seperable metric space,
see, e.g., \citep[Theorem 6.18]{V09}. Convergence with respect to
this distances is explained in the following remark, see also \citep[Theorem 6.9]{V09}.
\begin{Remark}\label{WASSERSTEIN:REMARK} Let $d\in\{d_{\log},d_{\kappa}\}$,
$(\rho_{n})_{n\in\N}\subset\mathcal{P}_{d}(D)$ and $\rho\in\mathcal{P}_{d}(D)$.
The following are equivalent
\begin{enumerate}
\item[(i)] $W_{d}(\rho_{n},\rho)\longrightarrow0$ as $n\to\infty$.
\item[(ii)] For each continuous function $f:D\longrightarrow\R$ with $|f(x)|\leq C_{f}(1+d(x,0))$,
it holds that
\[
\int\limits _{D}f(x)\rho_{n}(dx)\longrightarrow\int\limits _{D}f(x)\rho(dx),\qquad n\to\infty.
\]
\item[(iii)] $\rho_{n}\longrightarrow\rho$ weakly as $n\to\infty$, and
\[
\int\limits _{D}d(x,0)\rho_{n}(dx)\longrightarrow\int\limits _{D}d(x,0)\rho(dx),\qquad n\to\infty.
\]
\item[(iv)] $\rho_{n}\longrightarrow\rho$ weakly as $n\to\infty$, and
\[
\lim\limits _{R\to\infty}\limsup\limits _{n\to\infty}\int\limits _{D}d(x,0)\1_{\{d(x,0)\geq R\}}\rho_{n}(dx)=0.
\]
\end{enumerate}
\end{Remark} For simplicity of notation, we let $\mathcal{P}_{\kappa}(D)=\mathcal{P}_{d_{\kappa}}(D)$,
$\mathcal{P}_{\log}(D)=\mathcal{P}_{d_{\log}}(D)$, $W_{\kappa}=W_{d_{\kappa}}$,
and $W_{\log}=W_{d_{\log}}$. Then it is easy to see that $\mathcal{P}_{\kappa}(D)\subset\mathcal{P}_{\log}(D)$
and $W_{\log}\leq C_{\kappa}W_{\kappa}$, for some constant $C_{\kappa}>0$,
i.e., $W_{\kappa}$ is stronger then $W_{\log}$. The following is
our main result. \begin{Theorem}\label{THEOREM:01} Let $(a,\alpha,b,\beta,\nu,\mu)$
be admissible parameters. Suppose that $\beta$ has only eigenvalues with negative real parts, and
\begin{align}
\int\limits_{|\xi|>1}\log(|\xi|)\nu(d\xi)<\infty.\label{EQ:03}
\end{align}
Then $(P_t)_{t \geq 0}$ has a unique invariant distribution $\pi$ and the following assertions hold:
\begin{enumerate}
\item[(a)] $\pi\in\mathcal{P}_{\log}(D)$ and there exist constants $K,\delta>0$
such that, for all $\rho\in\mathcal{P}_{\log}(D)$,
\begin{align}
W_{\log}(P_{t}\rho,\pi)\leq K\min\left\{ e^{-\delta t},W_{\log}(\rho,\pi)\right\} +Ke^{-\delta t}W_{\log}(\rho,\pi),\qquad t\geq0.\label{EQ:14}
\end{align}
\item[(b)] If there exists $\kappa\in(0,1]$ satisfying
\begin{align}
\int\limits_{|\xi|>1}|\xi|^{\kappa}\nu(d\xi)<\infty,\label{FIRST:MOMENT}
\end{align}
then $\pi\in\mathcal{P}_{\kappa}(D)$ and there exists constants $K',\delta'>0$
such that, for all $\rho\in\mathcal{P}_{\kappa}(D)$,
\begin{align}
W_{\kappa}(P_{t}\rho,\pi)\leq K'W_{\kappa}(\rho,\pi)e^{-\delta't},\qquad t\geq0.\label{EQ:12}
\end{align}
\end{enumerate}
\end{Theorem} 
It is worthwhile to mention that to our knowledge a convergence rate solely under a $\log$-moment condition on the state-independent jump measure was not even obtained for one-dimensional CBI processes.
In order that $W_{\log}(P_{t}\rho,\pi)$ and $W_{\kappa}(P_{t}\rho,\pi)$
are well-defined, we need to show that $P_{t}\rho$ belongs to $\mathcal{P}_{\log}(D)$
or $\mathcal{P}_{\kappa}(D)$, respectively. This will be shown in
Section 5, where general moment estimates for affine processes are
studied. Using $P_{t}\delta_{x}=P_{t}(x,\cdot)$ combined with Remark
\ref{WASSERSTEIN:REMARK} we conclude the following.
\begin{Remark}
 Under the conditions of Theorem \ref{THEOREM:01}, there exist constants $\delta, K > 0$ such that
 \begin{align}\label{EQ:04}
  W_{d}(P_t(x,\cdot), \pi) \leq K e^{- \delta t} \left( 1 + W_{d}(\delta_x, \pi) \right), \qquad t \geq 0, \ x \in D,
 \end{align}
 where $d \in \{d_{\kappa}, d_{\log}\}$.
 Let $W_{d \wedge 1}$ be the Wasserstein distance given by
 \eqref{WASSERSTEIN:2} with $d$ replaced by $d \wedge 1$.
 Then similarly to Remark \ref{WASSERSTEIN:REMARK},
 convergence with respect to $W_{d \wedge 1}$ is equivalent to weak convergence of probability measures
 on $\mathcal{P}(D)$.
 Since $W_{d \wedge 1} \leq W_d$, we conclude from \eqref{EQ:04} that $P_{t}(x,\cdot)\longrightarrow\pi$
 weakly as $t\to\infty$ with exponential rate.
\end{Remark}
Let $X=(X_{t})_{t\geq0}$ be an affine process. For the parameter
estimation of affine models, see, e.g., \citep{BDLP13}, \citep{LM15}
and \citep{BBKP18}, it is often necessary to prove a Birkhoff ergodic
theorem, i.e.,
\begin{align}
\frac{1}{t}\int\limits _{0}^{t}f(X_{s})ds\longrightarrow\int\limits _{D}f(x)\pi(dx),\ \ t\to\infty\label{EQ:02}
\end{align}
holds almost surely for sufficiently many test functions $f$. Using
classical theory, see, e.g., \citep[Theorem 17.1.7]{MT09} and \citep{S17},
such convergence is implied by the ergodicity in the total variation
distance, i.e., by
\begin{align}
\lim\limits _{t\to\infty}\|P_{t}(x,\cdot)-\pi\|_{\mathrm{TV}}=0,\qquad x\in D,\label{TVERGODICITY}
\end{align}
where $\|\cdot\|_{\mathrm{TV}}$ denotes the total variation distance.
Unfortunately, it is typically a very difficult mathematical task
to prove \eqref{TVERGODICITY} even for particular models. An extension
of \eqref{EQ:02} applicable in the case where $P_{t}(x,\cdot)\longrightarrow\pi$
holds in the Wasserstein distance generated by the metric $d(x,\widetilde{x})=1\wedge|x-\widetilde{x}|$
was recently studied in \citep{S17}. Applying the main result of
\citep{S17} to the case of affine processes and using the fact that
each affine process can be obtained as a pathwise unique strong solution
to some stochastic equation with jumps (see Section 4), yields the
following corollary.
\begin{Corollary} Let $(a,\alpha,b,\beta,\nu,\mu)$
be admissible parameters. Suppose that $\beta$ has only eigenvalues
with negative real parts, and \eqref{EQ:03} is satisfied. Let $(X_{t})_{t\geq0}$
be the corresponding affine process constructed as the pathwise unique
strong solution on a complete probability space $(\Omega,\F,\P)$
in Section 4. Let $f\in L^{p}(D,\pi)$ for some $p\in[1,\infty)$, then \eqref{EQ:02} holds in $L^{p}(\Omega,\P)$.
\end{Corollary}
Although we have formulated \eqref{EQ:02} in continuous
time, the discrete-time analog can be obtained in
the same manner.

\subsection{Comparison with related works}

Consider an \textit{Ornstein-Uhlenbeck process} on $\R^n$, i.e., an affine process on state space $D=\R^{n}$ with admissible
parameters $(a,\alpha = 0,b,\beta,\nu,\mu = 0)$.
If $\beta$ has only eigenvalues with negative real parts and \eqref{EQ:03} is satisfied, then \citep{SY84} is applicable and hence the corresponding
Ornstein-Uhlenbeck process satisfies, for all $x\in\R^{n}$, $P_{t}(x,\cdot)\longrightarrow\pi$ weakly as $t\to\infty$. Under additional technical conditions on the measure $\nu$, it follows that the corresponding process also satisfies \eqref{TVERGODICITY} with exponential rate, see \citep{W12}. Since in view of Theorem \ref{THEOREM:01} the convergence (in the Wasserstein distance) has already exponential rate, we conclude that the additional restriction on $\nu$ imposed in \citep{W12} is only used to guarantee that convergence takes place in the stronger total variation distance, i.e., it is not necessary for the speed of convergence.

Consider a subcritical \textit{multi-type CBI process} on $\R_+^m$, i.e., an affine process on state space $D = \R_+^m$ for which the parameter $\beta$ has only eigenvalues with negative real parts. In dimension $m = 1$, Pinsky \cite{P72} announced (without proof) the existence of a limiting distribution under condition \eqref{EQ:03}. A proof of this fact was then given in \cite[Theorem 3.16]{KS08},
while in \citep[Theorem 3.20 and Corollary 3.21]{L11} it was shown 
that $P_{t}(x,\cdot)\longrightarrow\pi$
is equivalent to \eqref{EQ:03}. Some properties of the invariant distribution $\pi$ have been studied in \cite{KM12}.
In \citep{LM15} exponential ergodicity in total variation distance,
see \eqref{TVERGODICITY}, was established for one-dimensional
subcritical CBI processes with $\nu=0$, while some other related results for stochastic equations on $\R_+$ have been recently considered in \citep{2019arXiv190202833F}.
An extension of the techniques from \cite{LM15} 
to arbitrary dimension $m \geq 2$ is still an interesting open problem.  Recently, in \citep{MSV18} another approach for the exponential ergodicity in the total variation
distance for affine processes on cones, including multi-type CBI processes, was provided. Their techniques were closely related to stochastic stability of Markov processes in the sense of Meyn and Tweedie \citep{MT09}, see also the references therein. More precisely, it was shown that each subcritical
CBI process $X$ which is $\nu$-irreducible, aperiodic and has finite
second moments, where $\nu$ is a reference measure with its support
having non-empty interior, is exponentially ergodic in the total variation
distance. As such a result is formulated in a very general way, it becomes a delicate mathematical task to show that such conditions
are satisfied for CBI processes with jumps of infinite activity
or with degenerate diffusion components. Moreover, assuming that $X$
has at least finite second moments rules out some natural examples
as studied in \citep{LM15} for $m=1$ and in Section 2 of this work.
In contrast, our results can be applied in arbitrary dimension without the need to prove irreducibility or aperiodicity,
paying the price that we use the Wasserstein distance instead.
Let us mention that recently also asymptotic results for supercritical CBI processes have been obtained in \cite{KPR17, BPP18, BPP18b}.

Consider now the general case of an \textit{affine process} on the canonical state space $D = \R_+^m \times \R^n$. 
Based on the stability theory of Markov chains in the sense of Meyn and Tweedie the long-time behavior of some particular two-dimensional models on state space $D = \R_+ \times \R$ was studied in \cite{BDLP14, JKR17b}.These results have been further developed in \citep{GZ18} for arbitrary dimensions, where also functional limit theorems were obtained. In order to prove irreducibility and aperiodicity, the authors supposed that the diffusion compnent is non-degenerate and that $\nu$ and $\mu_{1},\dots,\mu_{m}$ are probability measures, i.e., the corresponding affine process has only jumps of finite variation.
Independently in \cite{JKR18} the following result was obtained.
\begin{Theorem}\citep{JKR18}\label{JKR:THEOREM}
Let $(a,\alpha,b,\beta,\nu,\mu)$ be admissible parameters. Suppose
that $\beta$ has only eigenvalues with negative real parts and
\eqref{EQ:03} is satisfied.
Then there exists a unique invariant distribution $\pi$ for $(P_{t})_{t\geq0}$. Moreover, $\pi$ has Laplace transform
\begin{align}
\int\limits _{D}e^{\langle u,x\rangle}\pi(dx)=\exp\left(\int\limits _{0}^{\infty}F(\psi(t,u))dt\right),\qquad u\in\mathcal{U},\label{INVARIANT}
\end{align}
and one has, for all $x\in D$, $P_{t}(x,\cdot)\longrightarrow\pi$
weakly as $t\to\infty$.
\end{Theorem}
The proof of Theorem \ref{JKR:THEOREM}
is based on a fine stability analysis of the Riccati equations \eqref{RICATTI}. Comparing with our main result Theorem \ref{THEOREM:01},
the authors have, in addition,
established a formula for the Laplace transform of $\pi$, but have not studied any convergence rate. We emphasize that the main aim of our Theorem \ref{THEOREM:01} is to establish the exponential convergence speed (\ref{EQ:14}) and (\ref{EQ:12}) with respect to the corresponding Wasserstein metrics. However, in the process of proving (\ref{EQ:14}) we also obtain the existence and uniqueness of an invariant distribution as a natural by-product.
Moreover, in Theorem \ref{THEOREM:01} and
Theorem \ref{JKR:THEOREM} existence
and uniqueness of an invariant distribution is shown
by essentially different techniques.

\subsection{Main idea of proof and structure of the work}

The proof of Theorem \ref{THEOREM:01} is divided in 4 steps as explained
below.

\textbf{Step 1.} Provide a stochastic description of conservative
affine processes. More precisely, in Section 3 we discuss a stochastic
equation for multi-type CBI processes and recall a comparison principle
due to \citep{BLP15}. In Section 4 we prove that each affine process
can be obtained as the pathwise unique strong solution $(X_t(x))_{t \geq 0}$ to a certain stochastic equation, where $x = (y,z) \in \R_+^m \times \R^n$ denotes the initial condition. 
The particular structure of this equation shows that the process takes the form $X_t(x) = (Y_t(y), Z_t(x))$, 
where $(Y_t(y))_{t \geq 0} \subset \R_+^m$ is a CBI process with initial condition $y$ and $(Z_t(x))_{t \geq 0}$ is an OU-type process with initial condition $z$ whose coefficients depend on the process $(Y_t(y))_{t \geq 0}$.

\textbf{Step 2.} Let $(X_{t})_{t\geq0}$ be an affine process. Based
on the stochastic equation from the first step, we study in Section
5 finiteness of the moments $\E(|X_{t}|^{\kappa})$ and $\E(\log(1+|X_{t}|))$.
Since the proofs in this section are rather standard, we only outline
the main steps, while technical details are postponed to the appendix.

\textbf{Step 3.} Let $\left(X_{t}(x)\right)_{t\ge0}$ and $\left(X_{t}(\widetilde{x})\right)_{t\ge0}$
be the affine processes with initial states $x, \ \widetilde{x}\in \R_{+}^{m}\times\R^{n}$,
respectively, obtained as the unique strong solutions to the stochastic
equation discussed in Section 4. Suppose that \eqref{FIRST:MOMENT}
is satisfied for $\kappa=1$. The following key estimate is proved
in Section 6:
\begin{align}\label{EQ:11}
\E(|X_{t}(x)-X_{t}(\widetilde{x})|)\leq Ke^{-\delta t}\left(\1_{\{n>0\}}|y-\widetilde{y}|^{1/2}+|x-\widetilde{x}|\right),\qquad t\geq0,
\end{align}
where $K,\delta>0$ are some constants.
Indeed, write $X_t(x) = (Y_t(y),Z_t(x))$ and $X_t(\widetilde{x}) = (Y_t(\widetilde{y}), Z_t(\widetilde{x}))$, respectively.
Using the comparison principle for the CBI component we prove that
\begin{align}\label{EQ:16}
 \E(|Y_t(x) - Y_t(\widetilde{x})|) \leq d |y - \widetilde{y}|e^{- \delta't},
\end{align}
where $\delta' > 0$ is some constant. 
From this and the particular structure of the stochastic equation 
solved by $(X_t(x))_{t \geq 0}$ and $(X_t(\widetilde{x}))_{t \geq 0}$
we then easily deduce \eqref{EQ:11}.
In the literature the proof of similar inequalities to \eqref{EQ:11} and \eqref{EQ:16} is often based on the construction of a successfull coupling which is typically a difficult task. 
In the framework of affine processes a surprisingly simple proof of such estimates is given in Section 6 by using monotone couplings as explained above.

\textbf{Step 4.} The results obtained in Steps 1 -- 3 provide us
all necessary tools to give a full proof of Theorem \ref{THEOREM:01}
in Section 7. For the sake of simplicity, we explain below how \eqref{EQ:12}
is shown. Estimate \eqref{EQ:14} can be obtained in the same way.
Using classical arguments, we may deduce assertion \eqref{EQ:12}
from the contraction estimate
\begin{align}
W_{\kappa}(P_{t}\rho,P_{t}\widetilde{\rho})\leq Ke^{-\delta t}W_{\kappa}(\rho,\widetilde{\rho}),\qquad t\geq0.\label{EQ:06}
\end{align}
Next observe that, by the convexity of the Wasserstein distance (see
Lemma \ref{WASSERSTEIN:1}) combined with \eqref{EQ:05}, property
\eqref{EQ:06} is implied by
\begin{align}
W_{\kappa}(P_{t}\delta_{x},P_{t}\delta_{\widetilde{x}})\leq Ke^{-\delta t}\left(\1_{\{n>0\}}|y-\widetilde{y}|^{1/2}+|x-\widetilde{x}|\right)^{\kappa},\qquad t\geq0.\label{EQ:07}
\end{align}
Let $(P_{t}^{0})_{t\geq0}$ be the transition semigroup for the affine
process with admissible parameters $(a=0,\alpha,b=0,\beta,m=0,\mu)$.
In view of \eqref{LAPLACE} one has $P_{t}(x,\cdot)=P_{t}^{0}(x,\cdot)\ast P_{t}(0,\cdot)$,
where $\ast$ denotes the usual convolution of measures. A similar
decomposition for affine processes was also used in \citep{JKR18}.
Applying now Lemma \ref{WASSERSTEIN} and the Jensen inequality gives
\begin{align*}
W_{\kappa}(P_{t}\delta_{x},P_{t}\delta_{\widetilde{x}}) & \leq W_{\kappa}(P_{t}^{0}\delta_{x},P_{t}\delta_{\widetilde{x}})\\
 & \leq(W_{1}(P_{t}^{0}\delta_{x},P_{t}^{0}\delta_{\widetilde{x}}))^{\kappa}\leq K^{\kappa}e^{-\delta\kappa t}\left(\1_{\{n>0\}}|y-\widetilde{y}|^{1/2}+|x-\widetilde{x}|\right)^{\kappa},
\end{align*}
where the last inequality follows from Step 3 applied to $(P_{t}^{0})_{t\geq0}$.

\section{Examples}

\subsection{Anisotropic $(\gamma_{1},\gamma_{2})$-root process}

Let $Z_{1},Z_{2}$ be independent one-dimensional Lévy processes with
symbols
\[
\Psi_{j}(\xi)=\int\limits _{0}^{\infty}\left(e^{-\xi z}-1+\xi z\right)\frac{dz}{z^{1+\gamma_{j}}},\qquad\xi\ge0,\ j=1,2,
\]
where $\gamma_{1},\gamma_{2}\in(1,2)$. Let $S=(S_{1},S_{2})$ be
another $2$-dimensional Lévy process with symbol
\[
\Psi_{\nu}(\xi)=\int\limits _{\R_{+}^{2}}\left( e^{-\langle\xi,z\rangle}-1\right)\nu(dz),\qquad\xi\in\R_{+}^{2},
\]
where $\nu$ is a measure on $\R_{+}^{2}$ with $\nu\left(\left\{ 0\right\} \right)=0$
and
\[
\int\limits_{\R_{+}^{2}}\left(1\wedge|z|\right)\nu(dz)<\infty.
\]
Suppose that $Z$ and $S$ are independent. Applying the results
of \citep{BLP15} to this particular case shows that, for each $x\in\R_{+}^{2}$,
there exists a pathwise unique strong solution to
\begin{align*}
dX_{1}(t)=\left(b_{1}+\beta_{11}X_{1}(t)+\beta_{12}X_{2}(t)\right)dt+X_{1}(t-)^{1/\gamma_{1}}dZ_{1}(t)+dS_{1}(t),\\
dX_{2}(t)=\left(b_{2}+\beta_{21}X_{1}(t)+\beta_{22}X_{2}(t)\right)dt+X_{2}(t-)^{1/\gamma_{2}}dZ_{2}(t)+dS_{2}(t),
\end{align*}
This process is an affine process on $D=\R_{+}^{2}$ with admissible
parameters
\[
a=0,\ \ \alpha_{1}=\alpha_{2}=0,\ \ b=\begin{pmatrix}b_{1}\\
b_{2}
\end{pmatrix},\ \ \beta=\begin{pmatrix}\beta_{11} & \beta_{12}\\
\beta_{21} & \beta_{22}
\end{pmatrix}
\]
and corresponding Lévy measures $\nu$,
\[
\mu_{1}(d\xi)=\frac{d\xi_{1}}{\xi_{1}^{1+\gamma_{1}}}\otimes\delta_{0}(d\xi_{2}),\ \ \mu_{2}(d\xi)=\delta_{0}(d\xi_{1})\otimes\frac{d\xi_{2}}{\xi_{2}^{1+\gamma_{2}}}.
\]
Applying our main result to this particular case gives the following.
\begin{Corollary} If $\beta$ has only eigenvalues with negative
real parts and $\nu$ satisfies
\[
\int\limits _{|\xi|>1}\log(|\xi|)\nu(d\xi)<\infty,
\]
then the assertions of Theorem \ref{THEOREM:01} are true. \end{Corollary}
Convergence in total variation distance for a similar one-dimensional
model was studied in \citep{LM15}. Similar two-dimensional processes
were also studied in \citep{BDLP14} and \citep{JKR17}. In view of
our main result Theorem \ref{THEOREM:01}, it is straightforward
to extend this model to arbitrary dimension $d\geq 2$, with possibly
non-vanishing diffusion part and more general driving noise of Lévy
type.

\subsection{Stochastic volatility model}

Let $D=\R_{+}\times\R$, i.e., $m=n=1$. Let $(V,Y)$ be the unique
strong solution to
\begin{align*}
dV(t) & =(b_{1}+\beta_{11}V(t))dt+\sqrt{V(t)}dB_{1}(t)+dJ_{1}(t),\\
dY(t) & =(b_{2}+\beta_{22}Y(t))dt+\sqrt{V(t)}\left(\rho dB_{1}(t)+\sqrt{1-\rho^{2}}dB_{2}(t)\right)+dJ_{2}(t)
\end{align*}
where $b_{1}\geq0$, $b_{2}\in\R$, $\beta_{11},\beta_{22}\in\R$,
$\rho\in(-1,1)$ is the correlation coefficient, $B=(B_{1},B_{2})$
is a two-dimensional Brownian motion, $J_{1}$ is a one-dimensional
Lévy subordinator with Lévy measure $\nu_{1}$, and $J_{2}$ a one-dimensional
Lévy process with Lévy measure $\nu_{2}$. Suppose that $B,\ J_{1}$
and $J_{2}$ are mutually independent. It is not difficult to see
that $(V,Y)$ is an affine process with admissible parameters
\[
a=0,\ \ \alpha_{1}=\begin{pmatrix}1 & \rho\\
\rho & 1
\end{pmatrix},\ \ b=\begin{pmatrix}b_{1}\\
b_{2}
\end{pmatrix},\ \ \beta=\begin{pmatrix}\beta_{11} & 0\\
0 & \beta_{22}
\end{pmatrix}
\]
and measures
\[
\nu(d\xi)=\nu_{1}(d\xi_{1})\otimes\delta_{0}(d\xi_{2})+\delta_{0}(d\xi_{1})\otimes \nu_{2}(d\xi_{2}),\ \ \mu_{1}=\mu_{2}=0.
\]
Then we obtain the following. \begin{Corollary} If $\beta_{11},\beta_{22}<0$
and
\[
\int\limits _{(1,\infty)}\log(\xi_{1})\nu_{1}(d\xi_{1})+\int\limits _{|\xi_{2}|>1}\log(|\xi_{2}|)\nu_{2}(d\xi_{2})<\infty,
\]
then the assertions of Theorem \ref{THEOREM:01} are true. \end{Corollary}
It is straightforward to extend this model to higher dimensions and
more general driving noises.

\section{Stochastic equation for multi-type CBI processes}

In this section we recall some results for the particular case of
multi-type CBI processes, i.e. affine processes on state space $D=\R_{+}^{m}$
(that is, $n=0$). For further references and additional explanations
we refer to \citep{BLP15} and \citep{BPP18b}. Let $(\Omega,\F,\P)$
be a complete probability space rich enough to support the following
objects:
\begin{enumerate}
\item[(B1)] A $m$-dimensional Brownian motion $\left(W_{t}\right)_{t\ge0}:=(W_{t,1},\dots,W_{t,m})_{t\ge0}$.
\item[(B2)] A Poisson random measure $M_{I}(ds,d\xi)$ on $\R_{+}\times\R_{+}^{m}$
with compensator $\widehat{M}_{I}(ds,d\xi)=ds\nu_{I}(d\xi)$,
where $\nu_{I}$ is a Borel measure supported on $\R_{+}^{m}$ satisfying
\[
\nu_{I}(\{0\})=0,\qquad\int\limits _{\R_{+}^{m}}(1\wedge|\xi|)\nu_{I}(d\xi)<\infty.
\]
\item[(B3)] Poisson random measures $N_{1}^{I}(ds,d\xi,dr),\dots,N_{m}^{I}(ds,d\xi,dr)$
on $\R_{+}\times\R_{+}^{m}\times\R_{+}$ with compensators $\widehat{N}_{i}^{I}(ds,d\xi,dr)=ds\mu_{i}^{I}(d\xi)dr$,
$i\in I$, where $\mu_{1}^{I},\dots,\mu_{m}^{I}$ are Borel measures
supported on $\R_{+}^{m}$ satisfying
\[
\mu_{i}^{I}(\{0\})=0,\ \ \int\limits _{\R_{+}^{m}}\left(|\xi|\wedge|\xi|^{2}+\sum\limits _{j\in\{1,\dots,m\}\backslash\{i\}}\xi_{j}\right)\mu_{i}^{I}(d\xi)<\infty,\qquad i\in I.
\]
\end{enumerate}
The objects $W,M_{I},N_{1}^{I},\dots,N_{m}^{I}$ are supposed to be
mutually independent. Let $\widetilde{M}_{I}(ds,d\xi)=M_{I}(ds,d\xi)-\widehat{M}_{I}(ds,d\xi)$
and $\widetilde{N}_{i}^{I}(ds,d\xi,dr)=N_{i}^{I}(ds,d\xi,dr)-\widehat{N}_{i}^{I}(ds,d\xi,dr)$
be the corresponding compensated Poisson random measures. 
Here and below we consider the natural augmented filtration generated by $W, M_I, N_1^I, \dots, N_m^I$.
Finally let
\begin{enumerate}
\item[(a)] $b\in\R_{+}^{m}$.
\item[(b)] $\beta=(\beta_{ij})_{i,j\in I}$ such that $\beta_{ji}-\int_{\R_{+}^{m}}\xi_{j}\mu_{i}^{I}(d\xi)\geq0$,
for $i\in I$ and $j\in I\backslash\{i\}$.
\item[(c)] A matrix $\sigma(y)=\mathrm{diag}(\sqrt{2c_{1}y_{1}},\cdots,\sqrt{2c_{m}y_{m}})\in\R^{m\times m}$,
where $c_{1},\dots,c_{m}\geq0$.
\end{enumerate}
For $y\in\R_{+}^{m}$, consider the stochastic equation
\begin{align}
Y_{t} & =y+\int\limits _{0}^{t}\left(b+\widetilde{\beta}Y_{s}\right)ds+\int\limits _{0}^{t}\sigma(Y_{s})dW_{s}+\int\limits _{0}^{t}\int\limits _{\R_{+}^{m}}\xi M_{I}(ds,d\xi)\label{SDE:3}\\
 & \ \ \ +\sum\limits _{i=1}^{m}\int\limits _{0}^{t}\int\limits _{|\xi|\leq1}\int\limits _{\R_{+}}\xi\1_{\{r\leq Y_{s-,i}\}}\widetilde{N}_{i}^{I}(ds,d\xi,dr)+\sum\limits _{i=1}^{m}\int\limits _{0}^{t}\int\limits _{|\xi|>1}\int\limits _{\R_{+}}\xi\1_{\{r\leq Y_{s-,i}\}}N_{i}^{I}(ds,d\xi,dr),\nonumber
\end{align}
where $\widetilde{\beta}_{ji}=\beta_{ji}-\int_{|\xi|>1}\xi_{j}\mu_{i}^{I}(d\xi)$.
Pathwise uniqueness for a slightly more complicated equation was recently
obtained in \citep{BLP15}, while \eqref{SDE:3} in this form appeared
first in \citep{BPP18b}. The following is essentially due to \citep{BLP15}.
\begin{Proposition}\label{PROP:00}
Let $(b,\beta,\sigma)$
be as in (a) -- (c), and consider objects $W,M_{I},N_{1}^{I},\dots,N_{m}^{I}$
that are given in (B1) -- (B3).
Then the following assertions hold:
\begin{enumerate}
\item[(a)] For each $y\in\R_{+}^{m}$, there exists a pathwise unique strong
solution $Y=(Y_{t})_{t\geq0}$ to \eqref{SDE:3}.
\item[(b)] Let $Y$ be any solution to \eqref{SDE:3}. Then
$Y$ is a multi-type CBI process starting from $y$, and the generator
$L_{Y}$ of $Y$ is of the following form: for $f\in C_{c}^{2}(\R_{+}^{m}),$
\begin{align*}
(L_{Y}f)(y) & =(b+\beta y,\nabla f(y))+\sum\limits _{i=1}^{m}c_{i}y_{i}\frac{\partial^{2}f(y)}{\partial y_{i}^{2}}+\int\limits _{\R_{+}^{m}}\left(f(y+\xi)-f(y)\right)\nu_{I}(d\xi)\\
 & \ \ \ +\sum\limits _{i=1}^{m}y_{i}\int\limits _{\R_{+}^{m}}\left(f(y+\xi)-f(y)-(\xi,\nabla f(y))\right)\mu_{i}^{I}(d\xi).
\end{align*}
Conversely, given any multi-type CBI process $\widetilde{Y}$ with
generator $L_{Y}$ and starting point $y$, we can
find a solution $Y$ to \eqref{SDE:3} such that $Y$ and $\widetilde{Y}$
have the same law.
\end{enumerate}
\end{Proposition} The proof of the pathwise uniqueness is based on
a comparison principle for multi-type CBI processes, see \citep[Lemma 4.2]{BLP15}.
This comparison principle is stated below.
\begin{Lemma}\citep[Lemma 4.2]{BLP15}\label{COMPARISON}
Let $(Y_{t})_{t\geq0}$ be a weak solution to \eqref{SDE:3} with
parameters $(b,\beta,\sigma)$, let $(Y_{t}')_{t\geq0}$ be another
weak solution to \eqref{SDE:3} with parameters $(b',\beta,\sigma)$,
where $(b,\beta,\sigma)$ and $(b',\beta,\sigma)$
satisfy (a) -- (c).
Both solutions are supposed to be defined on
the same probability space and with respect to the
same noises $W,M_{I},N_{1}^{I},\dots,N_{m}^{I}$ that satisfy (B1) -- (B3).
Suppose that, for all $j\in\{1,\dots,m\}$, $y_{j}\leq y_{j}'$
and $b_{j}\leq b_{j}'$. Then
\[
\P(Y_{j,t}\leq Y_{j,t}',\ \ \forall j\in\{1,\dots,m\},\ \forall t\geq0)=1.
\]
\end{Lemma}

\section{Stochastic equation for affine processes}

Below we show that any affine process can also be obtained as the
pathwise unique strong solution to a certain stochastic equation.
In the two-dimensinoal case
$D = \R_+ \times \R$ such a result was first obtained in \cite{DL06}.
Indepedently, the case of affine diffusions on the canoncical state space
$D = \R_+^m \times \R^n$ (i.e., processes without jumps) was studied in \cite{FM09}. The main obstacle there is related with the diffusion component which is degenerate at the boundary but also has a nontrival structure in higher dimensions. In order to take this into account we use, compared to \cite{FM09}, another representation of the diffusion matrix (see (A0) and (A1) below). Such a representation is used to decompose the corresponding affine process into a CBI and an OU component which are then treated seperately. Consequently, based on the avaliable results for CBI processes, the proofs in this section become relatively simple.

Let $(a,\alpha,b,\beta,\nu,\mu)$ be admissible parameters. For
the parameters $a$ and $\alpha=(\alpha_{1},\dots,\alpha_{m})$ consider
the following objects:
\begin{enumerate}
\item[(A0)] An $n\times n$-matrix $\sigma_{a}$ such that $\sigma_{a}\sigma_{a}^{\top}=a_{JJ}$.
\item[(A1)] Matrices $\sigma_{1},\dots,\sigma_{m}\in\R^{d\times d}$ such that,
for all $j\in I$, $\sigma_{j}\sigma_{j}^{\top}=\alpha_{j}$
and
\begin{align}
\sigma_{j}=\begin{pmatrix}\sigma_{j,II} & 0\\
\sigma_{j,JI} & \sigma_{j,JJ}
\end{pmatrix},\qquad(\sigma_{j,II})_{kl}=\delta_{kj}\delta_{lj}\alpha_{j,jj}^{1/2}.\label{EQ:00}
\end{align}
\end{enumerate}
Let us remark the following.
\begin{Remark}
\begin{enumerate}
\item[(i)] The first condition is simple to check. Indeed, by definition, one
has $a=\begin{pmatrix}0 & 0\\
0 & a_{JJ}
\end{pmatrix}\in S_{d}^{+}$, thus $a_{JJ}$ is symmetric and positive semidefinite. Hence $\sigma_{a}$
denotes the non-negative square root of $a_{JJ}$.
\item[(ii)] Concerning the second condition, recall that $\alpha_{j}\in S_{d}^{+}$
and hence $\alpha_{j,II}$ is positive semidefinite.
Moreover, by definition of admissible parameters, $\alpha_{j,II}$
is everywhere zero except at the entry $(j,j)$. Hence $\alpha_{j,jj}^{1/2}$
is well-defined. Existence of $\sigma_{j}$ satisfying \eqref{EQ:00}
follows from the characterization of positive semidefiniteness for
symmetric block matrices, see, e.g., \citep[Theorem 16.1]{G11}. The
latter result is based on pseudo-inverses and properties of the Schur-complement
for block matrices.
\end{enumerate}
\end{Remark}
Below we describe the noises appearing
in the stochastic equation for affine processes. Let $(\Omega,\F,\P)$
be a complete probability space rich enough to support the following
objects:
\begin{enumerate}
\item[(A2)] A $n$-dimensional Brownian motion $B=(B_{t})_{t\geq0}$.
\item[(A3)] For each $i\in I$, a $d$-dimensional Brownian motion $W^{i}=(W_{t}^{i})_{t\geq0}$.
\item[(A4)] A Poisson random measure $M(ds,d\xi)$ with compensator $\widehat{M}(ds,d\xi)=ds\nu(d\xi)$
on $\R_{+}\times D$.
\item[(A5)] For each $i\in I$, a Poisson random measure $N_{i}(ds,d\xi,dr)$
with compensator $\widehat{N}_{i}(ds,d\xi,dr)=ds\mu_{i}(d\xi)dr$
on $\R_{+}\times D\times\R_{+}$.
\end{enumerate}
We suppose that all objects $B,W^{1},\dots,W^{m},M,N_{1},\dots,N_{m}$
are mutually independent. Denote by $\widetilde{M}(ds,d\xi)=M(ds,d\xi)-\widehat{M}(ds,d\xi)$
and $\widetilde{N}_{i}(ds,d\xi,dr)=N_{i}(ds,d\xi,dr)-\widehat{N}_{i}(ds,d\xi,dr)$,
$i\in I$, the corresponding compensated Poisson random measures. Here and below we consider the natural
augmented filtration generated by these noise terms.
For $x\in D$, consider the stochastic equation
\begin{align}
X_{t} & =x+\int\limits _{0}^{t}\left(\widetilde{b}+\widetilde{\beta}X_{s}\right)ds+\sqrt{2}\begin{pmatrix}0\\
\sigma_{a}B_{t}
\end{pmatrix}+\sum\limits _{i\in I}\int\limits _{0}^{t}\sqrt{2X_{s,i}}\sigma_{i}dW_{s}^{i}\label{SDE}\\
 & \ \ \ +\int\limits _{0}^{t}\int\limits _{|\xi|\leq1}\xi\widetilde{M}(ds,d\xi)+\int\limits _{0}^{t}\int\limits _{|\xi|>1}\xi M(ds,d\xi)\nonumber \\
 & \ \ \ +\sum\limits _{i\in I}\int\limits _{0}^{t}\int\limits _{|\xi|\leq1}\int\limits _{\R_{+}}\xi\1_{\{r\leq X_{s-,i}\}}\widetilde{N}_{i}(ds,d\xi,dr)+\sum\limits _{i\in I}\int\limits _{0}^{t}\int\limits _{|\xi|>1}\int\limits _{\R_{+}}\xi\1_{\{r\leq X_{s-,i}\}}N_{i}(ds,d\xi,dr),\nonumber
\end{align}
where $\widetilde{b}$ and $\widetilde{\beta}=(\widetilde{b}_{ki})_{k,i\in\{1,\dots,d\}}$
are, for $i,k\in\{1,\dots,d\}$, given by
\begin{align}
\widetilde{b}_{i}=b_{i}+\1_{I}(i)\int\limits _{|\xi|\leq1}\xi_{i}\nu(d\xi),\qquad\widetilde{\beta}_{ki}=\beta_{ki}-\1_{I}(i)\int\limits _{|\xi|>1}\xi_{k}\mu_{i}(d\xi).\label{COMP:00}
\end{align}
Note that we have changed the drift coefficients to $\widetilde{b}$
and $\widetilde{\beta}$ in order to change the compensators in the
stochastic integrals. Such change is, under the given moment conditions
on $\mu=(\mu_{1},\dots,\mu_{m})$, always possible and does not affect
our results. Concerning existence and uniqueness for \eqref{SDE},
we obtain the following.
\begin{Theorem}\label{THEOREM:00} Let $(a,\alpha,b,\beta,\nu,\mu)$
be admissible parameters. Then, for each $x\in D$, there exists a
pathwise unique $D$-valued strong solution $X=(X_{t})_{t\geq0}$
to \eqref{SDE}.
\end{Theorem}
This result will be proved later in
this Section. Let us first relate \eqref{SDE} to affine processes.
\begin{Proposition}
Let $(a,\alpha,b,\beta,\nu,\mu)$
be admissible parameters. Then each solution $X$ to \eqref{SDE}
is an affine process with admissible parameters $(a,\alpha,b,\beta,\nu,\mu)$
and starting point $x$.
\end{Proposition}
\begin{proof} Let
$X$ be a solution to \eqref{SDE} and $f\in C_{c}^{2}(D)$. Applying
the Itô formula shows that
\[
M_{f}(t):=f(X_{t})-f(x)-\int\limits _{0}^{t}(Lf)(X_{s})ds,\ \ t\geq0
\]
is a local martingale. Note that $Lf$ is bounded.
Hence
\[
\E(\sup\limits _{s\in[0,t]}|M_{f}(t)|)\leq2\|f\|_{\infty}+\int\limits _{0}^{t}\E(|Lf(X_{s})|)ds\leq2\|f\|_{\infty}+t\|Lf\|_{\infty}<\infty,\qquad t\geq0,
\]
and we conclude that $(M_{f}(t))_{t\geq0}$ is a true martingale.
It follows from Remark \ref{COROLLARY:00} that $X$
is an affine process with admissible parameters $(a,\alpha,b,\beta,\nu,\mu)$.
 \end{proof}
 The rest of this section is devoted to the proof of
Theorem \ref{THEOREM:00}. As often in the theory of stochastic equations,
existence of weak solutions is the easy part.
\begin{Lemma}\label{WEAK:EXISTENCE}
Let $(a,\alpha,b,\beta,\nu,\mu)$ be admissible parameters. Then, for
each $x\in D$, there exists a weak solution $X$ to \eqref{SDE}.
\end{Lemma}
\begin{proof}
Since existence of a solution to the martingale
problem with sample paths in the Skorokhod space over $D$ is known,
the assertion is a consequence of \citep{K11}, namely, the equivalence
between weak solutions to stochastic equations and martingale problems.
Alternatively, following \citep[p.993]{DFS03} we can show that each
solution to the martingale problem with generator $L$ and domain
$C_{c}^{2}(D)$ is a semimartingale and compute its
semimartingale characteristics (see \citep[Theorem 2.12]{DFS03}).
The assertion is then a consequence of the equivalence between weak
solutions to stochastic equations and semimartingales (see \citep[Chapter III, Theorem 2.26]{JS03}).
\end{proof} In view of the Yamada-Watanabe Theorem (see \citep{BLP15b}),
Theorem \ref{THEOREM:00} is proved, provided we can show pathwise
uniqueness for \eqref{SDE}. For this purpose we rewrite \eqref{SDE}
into its components $X=(Y,Z)$, where $Y\in\R_{+}^{m}$ and $Z\in\R^{n}$.
Introduce the notation $\xi=(\xi_{I},\xi_{J})\in D$, where $\xi_{I}=(\xi_{i})_{i\in I}$
and $\xi_{J}=(\xi_{j})_{j\in J}$. Moreover, let $W_{s}^{i}=(W_{s,I}^{i},W_{s,J}^{i})$
and write for the initial condition $x=(y,z)\in D$. Finally, let
$e_{1},\dots,e_{d}$ denote the canonical basis vectors in $\R^{d}$.
Then \eqref{SDE} is equivalent to the system of equations
\begin{align}
Y_{t} & =y+\int\limits _{0}^{t}\left(b_{I}+\widetilde{\beta}_{II}Y_{s}\right)ds+\sum\limits _{i\in I}e_{i}\int\limits _{0}^{t}\sqrt{2\alpha_{i,ii}Y_{s,i}}dW_{s,i}^{i}+\int\limits _{0}^{t}\int\limits _{D}\xi_{I}M(ds,d\xi)\label{SDE:1}\\
 & \ \ \ +\sum\limits _{i\in I}\int\limits _{0}^{t}\int\limits _{|\xi|\leq1}\int\limits _{\R_{+}}\xi_{I}\1_{\{r\leq Y_{s-,i}\}}\widetilde{N}_{i}(ds,d\xi,dr)+\sum\limits _{i\in I}\int\limits _{0}^{t}\int\limits _{|\xi|>1}\int\limits _{\R_{+}}\xi_{I}\1_{\{r\leq Y_{s-,i}\}}N_{i}(ds,d\xi,dr),\nonumber \\
Z_{t} & =z+\int\limits _{0}^{t}\left(b_{J}+\widetilde{\beta}_{JI}Y_{s}+\widetilde{\beta}_{JJ}Z_{s}\right)ds+\sqrt{2}\sigma_{a}B_{t}+\sum\limits _{i\in I}\int\limits _{0}^{t}\sqrt{2Y_{s,i}}\left(\sigma_{i,JI}dW_{s,I}^{i}+\sigma_{i,JJ}dW_{s,J}^{i}\right)\label{SDE:2}\\
 & \ \ \ +\int\limits _{0}^{t}\int\limits _{|\xi|\leq1}\xi_{J}\widetilde{M}(ds,d\xi)+\int\limits _{0}^{t}\int\limits _{|\xi|>1}\xi_{J}M(ds,d\xi)\nonumber \\
 & \ \ \ +\sum\limits _{i\in I}\int\limits _{0}^{t}\int\limits _{|\xi|\leq1}\int\limits _{\R_{+}}\xi_{J}\1_{\{r\leq Y_{s-,i}\}}\widetilde{N}_{i}(ds,d\xi,dr)+\sum\limits _{i\in I}\int\limits _{0}^{t}\int\limits _{|\xi|>1}\int\limits _{\R_{+}}\xi_{J}\1_{\{r\leq Y_{s-,i}\}}N_{i}(ds,d\xi,dr).\nonumber
\end{align}
Observe that the first equation for $Y$ does not involve $Z$. We
will show that \eqref{SDE:1} is precisely \eqref{SDE:3}, i.e., $Y$
is a multi-type CBI process and pathwise uniqueness holds for $Y$.
The second equation for $Z$ describes an OU-type process with random
coefficients depending on $Y$. If we regard $Y$ as conditionally fixed, then pathwise uniqueness for \eqref{SDE:2} is obvious. These
ideas are summarized in the next lemma.
\begin{Lemma}
Let $(a,\alpha,b,\beta,\nu,\mu)$
be admissible parameters. Then pathwise uniqueness holds for \eqref{SDE:1}
and \eqref{SDE:2}, and hence for \eqref{SDE}.
\end{Lemma}
\begin{proof}
Let $X=(Y,Z)$ and $X'=(Y',Z')$ be two solutions to
\eqref{SDE} with the same initial condition $x=(y,z)\in D$ both
defined on the same probability space. Then $Y$ and $Y'$ both satisfy
\eqref{SDE:1}. Let us show that \eqref{SDE:1} is precisely \eqref{SDE:3},
from which we deduce $\P(Y_{t}=Y_{t}',\ \ t\geq0)=1$. Set $\mathrm{pr}_{I}:D\longrightarrow\R_{+}^{m}$,
$\mathrm{pr}_{I}(x)=(x_{i})_{i\in I}$, and define
\begin{itemize}
\item A $m$-dimensional Brownian motion $W_{t}:=(W_{t,1}^{1},\dots,W_{t,m}^{m})$.
\item A Poisson random measure $M_{I}(ds,d\xi)$ on $\R_{+}\times\R_{+}^{m}$
by
\[
M_{I}([s,t]\times A)=M([s,t]\times\mathrm{pr}_{I}^{-1}(A)),
\]
where $0\leq s<t$ and $A\subset\R_{+}^{m}$ is a Borel set.
\item Poisson random measures $N_{1}^{I},\dots,N_{m}^{I}$ on $\R_{+}\times\R_{+}^{m}\times\R_{+}$
by
\[
N_{i}^{I}([s,t]\times A\times[c,d])=N_{i}([s,t]\times\mathrm{pr}_{I}^{-1}(A)\times[c,d]),\qquad i\in I,
\]
where $0\leq s<t$, $0\leq c<d$ and $A\subset\R_{+}^{m}$ is a Borel
set.
\end{itemize}
Note that the random objects $W,M_{I},N_{1}^{I},\dots,N_{m}^{I}$
are mutually independent. Moreover, it is not difficult to see that
$M_{I}$ and $N_{1}^{I},\dots,N_{m}^{I}$ have compensators
\[
\widehat{M}_{I}(ds,d\xi)=ds \nu_{I}(d\xi),\ \ \widehat{N}_{i}^{I}(ds,d\xi,dr)=ds\mu_{i}^{I}(d\xi)dr,\qquad i\in I,
\]
where $\nu_{I}=\nu\circ\mathrm{pr}_{I}^{-1}$ and $\mu_{i}^{I}=\mu_{i}\circ\mathrm{pr}_{I}^{-1}$.
Finally let $c_{j}=\alpha_{j,jj}$, $j\in\{1,\dots,m\}$, and
\[
\sigma(y)=\mathrm{diag}(\sqrt{2c_{1}y_{1}},\cdots,\sqrt{2c_{m}y_{m}})\in\R^{m\times m}.
\]
Then \eqref{SDE:1} is precisely \eqref{SDE:3}, and it follows from
Proposition \ref{PROP:00}.(a) that $\P(Y_{t}=Y_{t}',\ \ t\geq0)=1$.

It remains to prove pathwise uniqueness for \eqref{SDE:2}. Define,
for $l\geq1$, a stopping time $\inf\{t>0\ |\ \max\{|Z_{t}|,|Z_{t}'|\}>l\}$.
Since $Z$ and $Z'$ both satisfy \eqref{SDE:2} for the same $Y$,
we obtain
\begin{align}
Z_{t\wedge\tau_{l}}-Z_{t\wedge\tau_{l}}'=\int\limits _{0}^{t\wedge\tau_{l}}\widetilde{\beta}_{JJ}(Z_{s}-Z_{s}')ds\label{eq:Z_t wedge tau l}
\end{align}
and hence, for some constant $C>0$,
\[
\E(|Z_{t\wedge\tau_{l}}-Z_{t\wedge\tau_{l}}'|)\le C\int\limits _{0}^{t}\E(|Z_{s\wedge\tau_{l}}-Z_{s\wedge\tau_{l}}'|)ds.
\]
The Grownwall lemma gives $\P(Z_{t\wedge\tau_{l}}=Z'_{t\wedge\tau_{l}})=1$,
for all $t\geq0$ and $l\geq1$. Note that $Z$ and $Z'$ have no
explosion. Taking $l\to\infty$ proves the assertion. 
\end{proof}

\section{Moments for affine processes}

The stochastic equation introduced in Section 4 can be used to provide
a simple proof for the finiteness of moments for affine processes.
The following is our main result for this section. \begin{Proposition}\label{PROP:MOMENT1}
Let $(a,\alpha,b,\beta,\nu,\mu)$ be admissible parameters. For $x\in D$,
let $X$ be the unique solution to \eqref{SDE}.
\begin{enumerate}
\item[(a)] Suppose that there exists $\kappa>0$ such that
\[
\int\limits _{|\xi|>1}|\xi|^{\kappa}\mu_{i}(d\xi)+\int\limits _{|\xi|>1}|\xi|^{\kappa}\nu(d\xi)<\infty,\qquad i\in I.
\]
Then there exists a constant $C_{\kappa}>0$ (independent of $x$ and $X$) such that
\[
\E(|X_{t}|^{\kappa})\leq(1+|x|^{\kappa})e^{C_{\kappa}t},\qquad t\geq0.
\]
\item[(b)] Suppose that \eqref{EQ:03} is satisfied. Then there exists a constant
$C>0$ (independent of $x$ and $X$) such that
\[
\E(\log(1+|X_{t}|))\leq(1+\log(1+|x|))e^{Ct},\qquad t\geq0.
\]
\end{enumerate}
\end{Proposition} \begin{proof} Define $V_{1}(h)=(1+|h|^{2})^{\kappa/2}$
and $V_{2}(h)=\log(1+|h|^{2})$, where $h\in D$. Applying the Itô
formula for $V_{j}$, $j\in\{1,2\}$, gives
\begin{align}
V_{j}(X_{t}) & =V_{j}(x)+\int\limits _{0}^{t}\mathcal{A}_{j}(X_{s})ds+\mathcal{M}_{j}(t),\label{EQ:01}
\end{align}
where $(\mathcal{M}_{j}(t))_{t\geq0}$ and $\mathcal{A}_{j}(\cdot)$
are given by
\begin{align*}
\mathcal{A}_{j}(h) & =\langle\widetilde{b}+\beta h,\nabla V_{j}(h)\rangle+\sum\limits _{k,l=1}^{d}\left(a_{kl}+\sum\limits _{i=1}^{m}\alpha_{i,kl}x_{i}\right)\frac{\partial^{2}V_{j}(h)}{\partial h_{k}\partial h_{l}}\\
 & \ \ \ +\int\limits _{D}\left(V_{j}(h+\xi)-V_{j}(h)-\langle\xi,\nabla V_{j}(h)\rangle\1_{\{|\xi|\leq1\}}\right)\nu(d\xi)\\
 & \ \ \ +\sum\limits _{i=1}^{m}h_{i}\int\limits _{D}\left(V_{j}(h+\xi)-V_{j}(h)-\langle\xi,\nabla V_{j}(h)\rangle\right)\mu_{i}(d\xi),\\
\mathcal{M}_{j}(t) & =\sqrt{2}\int\limits _{0}^{t}\left\langle \nabla_{J}V_{j}(X_{s}),\sigma_{a}dB_{s,J}\right\rangle +\sum\limits _{i=1}^{m}\int\limits _{0}^{t}\sqrt{2X_{s,i}}\left\langle \nabla V_{j}(X_{s}),\sigma_{i}dW_{s}^{i}\right\rangle \\
 & \ \ \ +\int\limits _{0}^{t}\int\limits _{D}\left(V_{j}(X_{s-}+\xi)-V_{j}(X_{s-})\right)\widetilde{M}(ds,d\xi)\\
 & \ \ \ +\sum\limits _{i=1}^{m}\int\limits _{0}^{t}\int\limits _{D}\int\limits _{\R_{+}}\left(V_{j}(X_{s-}+\xi\1_{\{r\leq X_{s-,i}\}})-V_{j}(X_{s-})\right)\widetilde{N}_{i}(ds,d\xi,dr),
\end{align*}
where $\widetilde{b}$ was defined in \eqref{COMP:00}. Define, for
$l\geq1$, a stopping time
$\tau_{l}=\inf\{t\geq0\ |\ |X_{t}| > l \}$.
Then it is not difficult to see that $(\mathcal{M}_{j}(t\wedge\tau_{l}))_{t\geq0}$
is a martingale, for any $l\geq1$. Moreover, we will prove in the
appendix that there exists a constant $C>0$ such that
\begin{align}
\mathcal{A}_{j}(h)\leq C(1+V_{j}(h)),\qquad h\in D.\label{EQ:10}
\end{align}
Hence taking expectations in \eqref{EQ:01} gives
\[
\E(V_{j}(X_{t\wedge\tau_{l}}))\leq V_{j}(x)+C\int\limits _{0}^{t}\left(1+\E(V_{j}(X_{s\wedge\tau_{l}}))\right)ds.
\]
Applying the Gronwall lemma gives $\E(V_{j}(X_{t\wedge\tau_{l}}))\leq(V_{j}(x)+Ct)e^{Ct}\leq(1+V_{j}(x))e^{C't}$,
for all $t\geq0$ and some constant $C'>0$. Since $(X_{t})_{t\geq0}$
has cádlág paths and $C'$ is independent of $l$, we may take the
limit $l\to\infty$ and conclude the assertion by the lemma of Fatou.
\end{proof} We close this section with a formula for the first moment
of general affine processes. The particular case of multi-type CBI
processes was treated in \citep[Lemma 3.4]{BLP15}, while recursion
formulas for higher-order moments of multi-type CBI processes were
provided in \citep{BLP16}. \begin{Lemma}\label{MOMENT} Let $(a,\alpha,b,\beta,\nu,\mu)$
be admissible parameters and suppose that
\begin{align}
\int\limits _{|\xi|>1}|\xi|\nu(d\xi)<\infty.\label{EQ:08}
\end{align}
Let $(X_{t})_{t\geq0}$ be an affine process obtained from \eqref{SDE}
with $X_{0}=x\in D$. Then
\[
\E(X_{t})=e^{t\beta}x+\int\limits _{0}^{t}e^{s\beta}\overline{b}ds,
\]
where $\overline{b}_{i}=b_{i}+\int_{|\xi|>1}\xi_{i}\nu(d\xi)+\1_{I}(i)\int_{|\xi|\leq1}\xi_{i}\nu(d\xi)$.
$x=(y,z)\in\R_{+}^{m}\times\R^{n}$ and $X=(Y,Z)\in \R_{+}^{m}\times\R^{n}$,
then
\begin{align*}
\E(Y_{t}) & =e^{t\beta_{II}}y+\int\limits _{0}^{t}e^{s\beta_{II}}\overline{b}_{I}ds,\\
E(Z_{t}) & =e^{t\beta_{JJ}}z+\int\limits _{0}^{t}e^{s\beta_{JJ}}\overline{b}_{J}ds+\int\limits _{0}^{t}e^{\left(t-s\right)\beta_{JJ}}\beta_{JI}e^{s\beta_{II}}yds+\int\limits _{0}^{t}\int\limits _{0}^{s}e^{\left(t-s\right)\beta_{JJ}}\beta_{JI}e^{u\beta_{II}}\overline{b}_{I}duds.
\end{align*}
\end{Lemma} \begin{proof} First observe that, by definition of admissible
parameters and \eqref{EQ:08}, we may apply Proposition \ref{PROP:MOMENT1}
(a) and deduce that $X_{t}$ has finite first moment. Taking expectations
in \eqref{SDE} gives
\[
\E(X_{t})=x+\int\limits _{0}^{t}\left(\overline{b}+\beta\E(X_{s})\right)ds.
\]
Solving this equation gives the desired formula for $\E(X_{t})$.
Taking expectations in \eqref{SDE:3} (or \eqref{SDE:1}) gives
\[
\E(Y_{t})=y+\int\limits _{0}^{t}\left(\overline{b}_{I}+\beta_{II}\E(Y_{s})\right)ds,
\]
which implies the desired formula for $\E(Y_{t})$. Finally, taking
expectations in \eqref{SDE:2} gives
\[
\E(Z_{t})=z+\int\limits _{0}^{t}\left(\overline{b}_{J}+\beta_{JI}\E(Y_{s})+\beta_{JJ}\E(Z_{s})\right)ds.
\]
Solving this equation and using previous formula for $\E(Y_{s})$,
we obtain the assertion. \end{proof}

\section{Contraction estimate for trajectories of affine processes}

The following is our main estimate for this section. \begin{Proposition}\label{LEMMA:00}
Let $(a,\alpha,b,\beta,\nu,\mu)$ be admissible parameters, suppose
that \eqref{EQ:08} is satisfied, and assume that $\beta$ has only
eigenvalues with negative real parts. Let $x=(y,z),\widetilde{x}=(\widetilde{y},\widetilde{z})\in \R_{+}^{m}\times\R^{n}$,
and let $X(x)=(Y(y),Z(x))$ and $X(\widetilde{x})=(Y(\widetilde{y}),Z(\widetilde{x}))$
be the unique strong solutions to \eqref{SDE} with initial condition
$x$ and $\widetilde{x}$, respectively. Then there exist constants
$K,\delta,\delta'>0$ independent of $X(x)$ and $X(\widetilde{x})$
such that, for all $t\geq0$,
\begin{align}
\E(|Y_{t}(y)-Y_{t}(\widetilde{y})|) & \leq d|y-\widetilde{y}|e^{-\delta't},\label{EST:00}\\
\E(|X_{t}(x)-X_{t}(\widetilde{x})|) & \leq Ke^{-\delta t}\left(\1_{\{n>0\}}|y-\widetilde{y}|^{1/2}+|x-\widetilde{x}|\right).\label{EST:01}
\end{align}
\end{Proposition} \begin{proof} Let us first prove \eqref{EST:00}.
Note that $Y(y)$ and $Y(\widetilde{y})$ are multi-type CBI processes
with the same parameters. If $\widetilde{y}_{j}\leq y_{j}$ for all
$j\in\{1,\dots,m\}$, then we obtain from Lemma \ref{COMPARISON}
and Lemma \ref{MOMENT}
\begin{align*}
\E(|Y_{t}(y)-Y_{t}(\widetilde{y})|) & \leq\sum\limits _{j=1}^{m}\E(|Y_{t,j}(y)-Y_{t,j}(\widetilde{y})|)\\
 & =\sum\limits _{j=1}^{m}\E(Y_{t,j}(y)-Y_{t,j}(\widetilde{y}))\\
 & =\sum\limits _{j=1}^{m}\left(e^{t\beta_{II}}(y-\widetilde{y})\right)_{j}\leq\sqrt{d}|e^{t\beta_{II}}(y-\widetilde{y})|\leq\sqrt{d}e^{-\delta't}|y-\widetilde{y}|,
\end{align*}
where we have used that $\beta_{II}$ has only eigenvalues with negative
real parts (since $\beta$ has this property and $\beta_{IJ}=0$).
For general $y,\widetilde{y}$, let $y^{0},\dots,y^{m}\in\R_{+}^{m}$
be such that
\[
y^{0}:=y,\ \ y^{m}=\widetilde{y},\ \ y^{j}=\sum\limits _{k=1}^{j}e_{k}\widetilde{y}_{k}+\sum\limits _{k=j+1}^{m}e_{k}y_{k},\ \ j\in\{1,\dots,m-1\},
\]
where $e_{1},\dots,e_{m}$ denote the canonical basis vectors in $\R^{m}$.
Then, for each $j\in\{0,\dots,m-1\}$, the elements $y^{j},y^{j+1}$
are comparable in the sense that $y_{k}^{j}=y_{k}^{j+1}$ if $k\neq j+1$,
and either $y_{j+1}^{j}\leq y_{j+1}^{j+1}$ or $y_{j+1}^{j}\geq y_{j+1}^{j+1}$.
In any case, we obtain from the previous consideration
\begin{align*}
\E(|Y_{t}(y)-Y_{t}(\widetilde{y})|) & \leq\sum\limits _{j=0}^{m-1}\E(|Y_{t}(y^{j})-Y_{t}(y^{j+1})|)\\
 & \leq\sqrt{d}e^{-\delta't}\sum\limits _{j=0}^{m-1}|y^{j}-y^{j+1}|\\
 & =\sqrt{d}e^{-\delta't}\sum\limits _{j=0}^{m-1}|y_{j+1}-\widetilde{y}_{j+1}|\leq de^{-\delta't}|y-\widetilde{y}|,
\end{align*}
where we have used $|y^{j}-y^{j+1}|=|y_{j+1}-\widetilde{y}_{j+1}|$.
This completes the proof of \eqref{EST:00}.

If $n=0$, then \eqref{EST:01} is trivial. Suppose that $n>0$. Applying
the Itô formula to $e^{-t\beta}X_{t}(x)$ and $e^{-t\beta}X_{t}(\widetilde{x})$,
and then taking the difference, gives
\begin{align*}
X_{t}(x)-X_{t}(\widetilde{x}) & =e^{t\beta}(x-\widetilde{x})+\sum\limits _{i\in I}\int\limits _{0}^{t}e^{(t-s)\beta}\left(\sqrt{2X_{s,i}(x)}-\sqrt{2X_{s,i}(\widetilde{x})}\right)\sigma_{i}dW_{s}^{i}\\
 & \ \ \ +\sum\limits _{i\in I}\int\limits _{0}^{t}\int\limits _{D}\int\limits _{\R_{+}}e^{(t-s)\beta}\xi\left(\1_{\{r\leq X_{s-,i}(x)\}}-\1_{\{r\leq X_{s-,i}(\widetilde{x})\}}\right)\widetilde{N}_{i}(ds,d\xi,dr).
\end{align*}
Here and below we denote by $K>0$ a generic constant which may vary
from line to line. Moreover, we find $\delta_{0}>0$ and $\delta\in(0,\delta')$
such that
\begin{align}
|e^{t\beta}|^{2}\leq e^{-\delta_{0}t}\text{ and }\int\limits _{0}^{t}e^{-(t-s)\frac{\delta_{0}}{2}}e^{-\delta's}ds\leq Ke^{-2\delta t},\qquad t\geq0.\label{EQ:15}
\end{align}
The stochastic integral against the Brownian motion is estimated by
the BDG-inequality as follows
\begin{align*}
 & \ \E\left(\left|\int\limits _{0}^{t}e^{(t-s)\beta}\left(\sqrt{2X_{s,i}(x)}-\sqrt{2X_{s,i}(\widetilde{x})}\right)\sigma_{i}dW_{s}^{i}\right|\right)\\
 & \leq K\left(\int\limits _{0}^{t}\E\left(\left|e^{(t-s)\beta}\left(\sqrt{2X_{s,i}(x)}-\sqrt{2X_{s,i}(\widetilde{x})}\right)\sigma_{i}\right|^{2}\right)ds\right)^{1/2}\\
 & \leq K\left(\int\limits _{0}^{t}e^{-\delta_{0}(t-s)}\E(|X_{s,i}(x)-X_{s,i}(\widetilde{x})|)ds\right)^{1/2}\\
 & \leq K\left(\int\limits _{0}^{t}e^{-\delta_{0}(t-s)}e^{-\delta's}ds\right)^{1/2}|y-\widetilde{y}|^{1/2}\leq Ke^{-\delta t}|y-\widetilde{y}|^{1/2},
\end{align*}
where we have used \eqref{EST:00} and \eqref{EQ:15}. For the stochastic
integral against $\widetilde{N}_{i}$ we consider the cases $|\xi|\leq1$
and $|\xi|>1$ separately. For $|\xi|\leq1$ we apply first the BDG-inequality
and then the Jensen inequality to obtain, for each $i\in I$,
\begin{align*}
 & \ \E\left(\left|\int\limits _{0}^{t}\int\limits _{|\xi|\leq1}\int\limits _{\R_{+}}e^{(t-s)\beta}\xi\left(\1_{\{r\leq X_{s-,i}(x)\}}-\1_{\{r\leq X_{s-,i}(\widetilde{x})\}}\right)\widetilde{N}_{i}(ds,d\xi,dr)\right|\right)\\
 & \leq K\E\left(\left|\int\limits _{0}^{t}\int\limits _{|\xi|\leq1}\int\limits _{\R_{+}}|e^{(t-s)\beta}\xi|^{2}|\1_{\{r\leq X_{s-,i}(x)\}}-\1_{\{r\leq X_{s-,i}(\widetilde{x})\}}|^{2}N_{i}(dr,d\xi,ds)\right|^{1/2}\right)\\
 & \leq K\left(\int\limits _{0}^{t}\int\limits _{|\xi|\leq1}\int\limits _{\R_{+}}|e^{(t-s)\beta}\xi|^{2}\E(|\1_{\{r\leq X_{s-,i}(x)\}}-\1_{\{r\leq X_{s-,i}(\widetilde{x})\}}|^{2})dr\mu_{i}(d\xi)ds\right)^{1/2}\\
 & \leq K\left(\int\limits _{0}^{t}e^{-(t-s)\delta_{0}}\E(|X_{s,i}(x)-X_{s,i}(\widetilde{x})|)ds\right)^{1/2}\\
 & \leq K|y-\widetilde{y}|^{1/2}\left(\int\limits _{0}^{t}e^{-(t-s)\delta_{0}}e^{-\delta's}ds\right)^{1/2}\leq Ke^{-\delta t}|y-\widetilde{y}|^{1/2}.
\end{align*}
For $|\xi|>1$, we apply first the BDG-inequality and then use the
sub-additivity of $a\longmapsto a^{1/2}$ to obtain
\begin{align*}
 & \ \E\left(\left|\int\limits _{0}^{t}\int\limits _{|\xi|>1}\int\limits _{\R_{+}}e^{(t-s)\beta}\xi\left(\1_{\{r\leq X_{s-,i}(x)\}}-\1_{\{r\leq X_{s-,i}(\widetilde{x})\}}\right)\widetilde{N}_{i}(ds,d\xi,dr)\right|\right)\\
 & \leq K\E\left(\left|\int\limits _{0}^{t}\int\limits _{|\xi|>1}\int\limits _{\R_{+}}|e^{(t-s)\beta}\xi|^{2}|\1_{\{r\leq X_{s-,i}(x)\}}-\1_{\{r\leq X_{s-,i}(\widetilde{x})\}}|^{2}N_{i}(dr,d\xi,ds)\right|^{1/2}\right)\\
 & \leq K\int\limits _{0}^{t}\int\limits _{|\xi|>1}\int\limits _{\R_{+}}\E\left(|e^{(t-s)\beta}\xi||\1_{\{r\leq X_{s-,i}(x)\}}-\1_{\{r\leq X_{s-,i}(\widetilde{x})\}}|\right)dr\mu_{i}(d\xi)ds\\
 & \leq K\int\limits _{0}^{t}e^{-(t-s)\frac{\delta_{0}}{2}}\E(|X_{s,i}(x)-X_{s,i}(\widetilde{x})|)ds\\
 & \leq K|y-\widetilde{y}|\int\limits _{0}^{t}e^{-(t-s)\frac{\delta_{0}}{2}}e^{-\delta's}ds\leq Ke^{-2\delta t}|x-\widetilde{x}|,
\end{align*}
where we have used $|y-\widetilde{y}|\leq|x-\widetilde{x}|$. Collecting
all estimates proves the assertion. \end{proof}

\section{Proof of Theorem \ref{THEOREM:01}}

\subsection{The $\log$-Wasserstein estimate}

Based on the results of Section 6, we first deduce
the following estimate with respect to the $\log$-Wasserstein distance.
\begin{Proposition}\label{PROP:04}
Let $(P_{t})_{t\ge0}$ be the
transition semigroup with admissible parameters $(a,\alpha,b,\beta,\nu,\mu)$,
suppose that $\beta$ has only eigenvalues with negative real parts,
and \eqref{EQ:03} is satisfied. Then there exist constants $K,\delta>0$
such that, for any $\rho,\widetilde{\rho}\in\mathcal{P}_{\log}(D)$,
one has
\[
W_{\log}(P_{t}\rho,P_{t}\widetilde{\rho})\leq K\min\left\{ e^{-\delta t},W_{\log}(\rho,\widetilde{\rho})\right\} +Ke^{-\delta t}W_{\log}(\rho,\widetilde{\rho}),\qquad t\geq0.
\]
\end{Proposition}
\begin{proof} Let $\left(P_{t}^{0}(x,\cdot)\right)_{t\ge0}$
be the transition semigroup with admissible parameters $(a,\alpha,b=0,\beta,m=0,\mu)$
given by Theorem \ref{EXISTENCE:AFFINE}. Take $x=(y,z),\widetilde{x}=(\widetilde{y},\widetilde{z})\in \R_{+}^{m}\times\R^{n}$
and let $X^{0}(x)=(Y^{0}(y),Z^{0}(x))$ and $X^{0}(\widetilde{x})=(Y^{0}(\widetilde{y}),Z^{0}(\widetilde{x}))$,
respectively, be the corresponding affine processes obtained from
\eqref{SDE} with admissible parameters $(a=0,\alpha,b=0,\beta,m=0,\mu)$.
Since $X_{t}^{0}(x)$ has law $P_{t}^{0}(x,\cdot)$ and $X_{t}^{0}(\widetilde{x})$
has law $P_{t}^{0}(\widetilde{x},\cdot)$, there exist by Proposition
\ref{LEMMA:00} constants $K,\delta>0$ such that
\begin{align*}
W_{1}(P_{t}^{0}(x,\cdot),P_{t}^{0}(\widetilde{x},\cdot)) & \leq\E\left(\1_{\{n>0\}}|Y_{t}^{0}(y)-Y_{t}^{0}(\widetilde{y})|^{1/2}+|X_{t}^{0}(x)-X_{t}^{0}(\widetilde{x})|\right)\\
 & \leq\1_{\{n>0\}}\left(\E(|Y_{t}^{0}(y)-Y_{t}^{0}(\widetilde{y})|)\right)^{1/2}+\E(|X_{t}^{0}(x)-X_{t}^{0}(\widetilde{x})|)\\
 & \leq Ke^{-\delta t}\left(\1_{\{n>0\}}|y-\widetilde{y}|^{1/2}+|x-\widetilde{x}|\right).
\end{align*}
Next observe that, for $u\in\mathcal{U}$, one has
\begin{align*}
\int\limits _{D}e^{\langle u,x'\rangle}P_{t}^{0}(x,dx')=\exp\left(\langle x,\psi(t,u)\rangle\right),\qquad\int\limits _{D}e^{\langle u,x'\rangle}P_{t}(0,dx')=\exp\left(\phi(t,u)\right).
\end{align*}
Combining this with \eqref{LAPLACE} proves $P_{t}(x,\cdot)=P_{t}^{0}(x,\cdot)\ast P_{t}(0,\cdot)$,
where $\ast$ denotes the convolution of measures on $D$. Let us
now prove the desired $\log$-estimate. Using Lemma \ref{WASSERSTEIN}
from the appendix and then the Jensen inequality for the concave function
$\R_{+}\ni a\longmapsto\log(1+a)$, gives for some generic constant
$K>0$
\begin{align}
\notag W_{\log}(P_{t}\delta_{x},P_{t}\delta_{\widetilde{x}}) & \leq W_{\log}(P_{t}^{0}\delta_{x},P_{t}^{0}\delta_{\widetilde{x}})
\\ \notag &\leq\log(1+W_{1}(P_{t}^{0}\delta_{x},P_{t}^{0}\delta_{\widetilde{x}}))
\\ \label{EQ:09} &\leq \log\left(1+Ke^{-\delta t}\left(\1_{\{n>0\}}|y-\widetilde{y}|^{1/2}+|x-\widetilde{x}|\right)\right)
\\ \notag &\leq K\min\{e^{-\delta t},\log(1+\1_{\{n>0\}}|y-\widetilde{y}|^{1/2}+|x-\widetilde{x}|)\}
\\ \notag & \ \ \ +Ke^{-\delta t}\log\left(1+\1_{\{n>0\}}|y-\widetilde{y}|^{1/2}+|x-\widetilde{x}|\right),
\end{align}
where we have used, for $a,b\geq0$, the elementary inequality
\begin{align*}
\log(1+ab) & {\color{magenta}{\normalcolor \leq}{\normalcolor K\min\{\log(1+a),\log(1+b)\}+K\log(1+a)\log(1+b)}}\\
 & \leq K\min\{a,\log(1+b)\}+Ka\log(1+b),
\end{align*}
which is proved in the appendix.
Applying now Lemma \ref{WASSERSTEIN:1} from the appendix gives for
any $H\in\mathcal{H}(\rho,\widetilde{\rho})$
\begin{align*}
W_{\log}(P_{t}\rho,P_{t}\widetilde{\rho}) & \leq\int\limits _{D\times D}W_{\log}(P_{t}\delta_{x},P_{t}\delta_{\widetilde{x}})H(dx,d\widetilde{x})\\
 & \leq K\int\limits _{D\times D}\min\left\{ e^{-\delta t},\log(1+\1_{\{n>0\}}|y-\widetilde{y}|^{1/2}+|x-\widetilde{x}|)\right\} H(dx,d\widetilde{x})\\
 & \ \ \ +Ke^{-\delta t}\int\limits _{D\times D}\log(1+\1_{\{n>0\}}|y-\widetilde{y}|^{1/2}+|x-\widetilde{x}|)H(dx,d\widetilde{x})\\
 & \leq K\min\left\{ e^{-\delta t},\int\limits _{D\times D}\log(1+\1_{\{n>0\}}|y-\widetilde{y}|^{1/2}+|x-\widetilde{x}|)H(dx,d\widetilde{x})\right\} \\
 & \ \ \ +Ke^{-\delta t}\int\limits _{D\times D}\log(1+\1_{\{n>0\}}|y-\widetilde{y}|^{1/2}+|x-\widetilde{x}|)H(dx,d\widetilde{x}).
\end{align*}
Choosing $H$ as the optimal coupling of $(\rho,\widetilde{\rho})$,
i.e.,
\[
W_{\log}(\rho,\widetilde{\rho})=\int\limits _{D\times D}\log(1+\1_{\{n>0\}}|y-\widetilde{y}|^{1/2}+|x-\widetilde{x}|)H(dx,d\widetilde{x}),
\]
proves the assertion. \end{proof} Based on previous proposition,
the proof of Theorem \ref{THEOREM:01} is easy. It is given below.
\begin{Lemma}\label{LEMMA:01}
Let $(P_{t})_{t\ge0}$ be the transition
semigroup with admissible parameters $(a,\alpha,b,\beta,\nu,\mu)$.
Suppose that $\beta$ has only eigenvalues with negative real parts,
and \eqref{EQ:03} is satisfied. Then $(P_t)_{t \geq 0}$ has a unique invariant distribution
$\pi$. Moreover, this distribution belongs to $\mathcal{P}_{\log}(D)$
and, for any $\rho\in\mathcal{P}_{\log}(D)$, one has \eqref{EQ:14}.
\end{Lemma}
\begin{proof}
Let us first prove existence of an invariant distribution $\widetilde{\pi} \in \mathcal{P}_{\log}(D)$.
Observe that, by Proposition \ref{PROP:MOMENT1}, we easily
deduce that $P_{t}\mathcal{P}_{\log}(D)\subset\mathcal{P}_{\log}(D)$,
for any $t\geq0$. Fix any $\rho\in\mathcal{P}_{\log}(D)$ and let
$k,l\in\N$ with $k>l$. Then
\begin{align*}
W_{\log}(P_{k}\rho,P_{l}\rho) & \leq\sum\limits _{s=l}^{k-1}W_{\log}(P_{s}P_{1}\rho,P_{s}\rho)\\
 & \leq K\sum\limits _{s=l}^{k-1}\min\left\{ e^{-\delta s},W_{\log}(P_{1}\rho,\rho)\right\} +K\sum\limits _{s=l}^{k-1}e^{-s\delta}W_{\log}(P_{1}\rho,\rho).
\end{align*}
Since the right-hand side tends to zero as $k,l\to\infty$, we conclude
that $(P_{k}\rho)_{k\in\N}$ is a Cauchy sequence in $(\mathcal{P}_{\log}(D),W_{\log})$.
In particular, there exists a limit $\pi \in \mathcal{P}_{\log}(D)$,
i.e., $W_{\log}(P_{k}\rho,\pi)\longrightarrow0$ as $k\to\infty$.
Let us show that $\pi$ is an invariant distribution for
$P_{t}$. Indeed, take $h\geq0$, then
\begin{align*}
W_{\log}(P_{h}\pi,\pi) & \leq W_{\log}(P_{h}\pi,P_{h}P_{k}\rho)+W_{\log}(P_{k}P_{h}\rho,P_{k}\rho)+W_{\log}(P_{k}\rho,\pi)\\
 & \leq K\min\left\{ e^{-\delta h},W_{\log}(\pi,P_{k}\rho)\right\} +Ke^{-\delta h}W_{\log}(\pi,P_{k}\rho)\\
 & \ \ \ +K\min\left\{ e^{-\delta k},W_{\log}(P_{h}\rho,\rho)\right\} +Ke^{-\delta k}W_{\log}(P_{h}\rho,\rho)+W_{\log}(P_{k}\rho,\pi).
\end{align*}
Since $W_{\log}(P_{k}\rho,\pi)\longrightarrow0$ as $k\to\infty$,
we conclude that all terms tend to zero. Hence $W_{\log}(P_{h}\pi,\pi)=0$,
i.e., $P_{h}\pi = \pi$, for all $h\geq0$. Next we prove that $\pi$ is the unique invariant distribution.
Let $\pi_0,\pi_1$ be any two invariant distributions and
define $W_{\log}^{\leq 1}$ as in \eqref{WASSERSTEIN:2}
with $d_{\log}$ replaced by $d_{\log} \wedge 1$.
Then we obtain, for any $t \geq 0$ and all $x, \widetilde{x} \in D$, by the proof of Proposition \ref{PROP:04}
(see \eqref{EQ:09})
\begin{align*}
 W_{\log}^{\leq 1}(P_t(x,\cdot), P_t(\widetilde{x}, \cdot))
 &\leq 1 \wedge W_{\log}(P_t(x,\cdot), P_t(\widetilde{x}, \cdot))
 \\ &\leq 1 \wedge \log \left(1 + K e^{-\delta t}\left( \1_{\{n > 0\}}|y - \widetilde{y}| + |x-\widetilde{x}| \right) \right).
\end{align*}
Fix any $H \in \mathcal{H}(\pi_0, \pi_1)$,
then using the invariance of $\pi_0, \pi_1$ together
with the convexity of the Wasserstein distance gives
\begin{align*}
 W_{\log}^{\leq 1}(\pi_0, \pi_1) &= W_{\log}^{\leq 1}(P_t \pi_0, P_t \pi_1)
 \\ &\leq \int \limits_{D \times D}W_{\log}^{\leq 1}(P_t(x,\cdot), P_t(\widetilde{x}, \cdot))H(dx,d\widetilde{x})
 \\ &\leq \int \limits_{D \times D} \min\{1,  \log(1 + 2Ke^{- \delta t}|x - \widetilde{x}|) H(dx,d\widetilde{x}).
\end{align*}
By dominated convergence we deduce that the right-hand side tends to zero as $t \to \infty$ and hence $\pi_0 = \pi_1$. The last assertion can now be deduced from
\begin{align*}
W_{\log}(P_{t}\rho,\pi)=W_{\log}(P_{t}\rho,P_{t}\pi)\leq K\min\left\{ e^{-\delta t},W_{\log}(\rho,\pi)\right\} +Ke^{-\delta t}W_{\log}(\rho,\pi),
\end{align*}
where we have first used the invariance of $\pi$ and then Proposition
\ref{PROP:04}. \end{proof}

\subsection{The $\kappa$-Wasserstein estimate}

As before, we start with an estimate with respect
to the Wasserstein distance $W_{\kappa}$.
\begin{Proposition}\label{PROP:02}
Let $(P_{t})_{t\ge0}$ be the transition semigroup with admissible
parameters $(a,\alpha,b,\beta,\nu,\mu)$. Suppose that $\beta$ has
only eigenvalues with negative real parts, and \eqref{FIRST:MOMENT}
is satisfied for some $\kappa\in(0,1]$. Then there exist constants
$K,\delta>0$ such that, for any $\rho,\widetilde{\rho}\in\mathcal{P}_{\kappa}(D)$,
one has
\begin{align*}
W_{\kappa}(P_{t}\rho,P_{t}\widetilde{\rho})\leq Ke^{-\delta t}W_{\kappa}(\rho,\widetilde{\rho}),\qquad t\geq0.
\end{align*}
\end{Proposition}
\begin{proof}
Let $\left(P_{t}^{0}(x,\cdot)\right)_{t\ge0}$
be the transition semigroup with admissible parameters $(a=0,\alpha,b=0,\beta,m=0,\mu)$
given by Theorem \ref{EXISTENCE:AFFINE}. Arguing as in the proof
of Proposition \ref{PROP:04}, we obtain
\begin{align}
W_{1}(P_{t}^{0}(x,\cdot),P_{t}^{0}(\widetilde{x},\cdot))\leq Ke^{-\delta t}\left(\1_{\{n>0\}}|y-\widetilde{y}|^{1/2}+|x-\widetilde{x}|\right),\label{EQ:13}
\end{align}
and $P_{t}(x,\cdot)=P_{t}^{0}(x,\cdot)\ast P_{t}(0,\cdot)$.
Then we obtain from Lemma \ref{WASSERSTEIN} from the appendix
\begin{align*}
W_{\kappa}(P_{t}\delta_{x},P_{t}\delta_{\widetilde{x}})
&\leq W_{\kappa}(P_{t}^{0}\delta_{x},P_{t}^{0}\delta_{\widetilde{x}})
\\ &\leq\left(W_{1}(P_{t}^{0}\delta_{x},P_{t}^{0}\delta_{\widetilde{x}})\right)^{\kappa}\leq K^{\kappa}e^{-\delta\kappa t}\left(\1_{\{n>0\}}|y-\widetilde{y}|^{1/2}+|x-\widetilde{x}|\right)^{\kappa},
\end{align*}
where the second inequality follows from the Jensen inequality and
the third is a consequence of \eqref{EQ:13}. Using now Lemma \ref{WASSERSTEIN:1}
from the appendix, we conclude that
\begin{align*}
W_{\kappa}(P_{t}\rho,P_{t}\widetilde{\rho}) & \leq\inf\limits _{H\in\mathcal{H}(\rho,\widetilde{\rho})}\int\limits _{D\times D}W_{\kappa}(P_{t}\delta_{x},P_{t}\delta_{\widetilde{x}})H(dx,d\widetilde{x})\\
 & \leq K^{\kappa}e^{-\delta\kappa t}\inf\limits _{H\in\mathcal{H}(\rho,\widetilde{\rho})}\int\limits _{D\times D}\left(\1_{\{n>0\}}|y-\widetilde{y}|+|x-\widetilde{x}|\right)^{\kappa}H(dx,d\widetilde{x})\\
 & =K^{\kappa}e^{-\delta\kappa t}W_{\kappa}(\rho,\widetilde{\rho}).
\end{align*}
This proves the assertion. \end{proof} Based on previous proposition,
the proof of the $W_{\kappa}$-estimate in Theorem \ref{THEOREM:01}
can be deduced by exactly the same arguments as in Lemma \ref{LEMMA:01}.
So Theorem \ref{THEOREM:01} is proved.

\section{Appendix}

\subsection{Moment estimates for $V_{1}$ and $V_{2}$}

In this section we prove \eqref{EQ:10}. \begin{Lemma} Suppose that
the same conditions as in Proposition \ref{PROP:MOMENT1} (a) are
satisfied. Then there exists a constant $C=C_{\kappa}>0$ such that
\[
\mathcal{A}_{1}(x)\leq CV_{1}(x),\qquad x=(y,z)\in \R_{+}^{m}\times\R^{n}.
\]
\end{Lemma}
\begin{proof}
Observe that $\nabla V_{1}(x)= \kappa x(1+|x|^{2})^{\frac{\kappa-2}{2}}$.
Using $|x|\leq(1+|x|^{2})^{1/2}$ gives $|\nabla V_{1}(x)|\leq \kappa(1+|x|^{2})^{\frac{\kappa-1}{2}}$,
and hence we obtain for some generic constant $C=C_{\kappa}>0$
\[
(\widetilde{b}+\beta x,\nabla V_{1}(x))\leq C\left(1+|x|\right)|\nabla V_{1}(x)|\leq CV_{1}(x).
\]
For the second order term we first observe that, for $k,l\in\{1,\dots,d\}$,
\[
\frac{\partial^{2}V_{1}(x)}{\partial x_{k}\partial x_{l}}= \kappa(\kappa-2)x_{k}x_{l}(1+|x|^{2})^{\frac{\kappa-4}{2}}+\delta_{kl}\kappa(1+|x|^{2})^{\frac{\kappa-2}{2}},
\]
where $\delta_{kl}$ denotes the Kronecker-Delta symbol. Using $x_{k}x_{l}\leq\frac{x_{k}^{2}+x_{l}^{2}}{2}\leq|x|^{2}\leq(1+|x|^{2})$
gives $\left|\frac{\partial^{2}V_{1}(x)}{\partial x_{k}\partial x_{l}}\right|\leq C(1+|x|^{2})^{\frac{\kappa-2}{2}}$.
This implies that
\begin{align*}
\sum\limits _{k,l=1}^{d}\left(a_{kl}+\sum\limits _{i=1}^{m}\alpha_{i,kl}x_{i}\right)\frac{\partial^{2}V_{1}(x)}{\partial x_{k}\partial x_{l}}\leq C(1+|x|)(1+|x|^{2})^{\frac{\kappa-2}{2}}\leq CV_{1}(x).
\end{align*}
Let us now estimate the integrals against $m$ and $\mu_{1},\dots,\mu_{m}$.
Consider first the case $|\xi|>1$. The mean value theorem gives
\begin{align*}
V_{1}(x+\xi)-V_{1}(x) & =\int\limits _{0}^{1}\left\langle \xi,\nabla V_{1}(x+t\xi)\right\rangle dt\\
 & = \kappa\int\limits _{0}^{1}\left\langle \xi,x+t\xi\right\rangle (1+|x+t\xi|^{2})^{\frac{\kappa-2}{2}}dt
 \leq \kappa|\xi|\int\limits _{0}^{1}(1+|x+t\xi|^{2})^{\frac{\kappa-1}{2}}dt,
\end{align*}
where we have used $\left\langle \xi,x+t\xi\right\rangle \leq|\xi||x+t\xi|\leq|\xi|(1+|x+t\xi|^{2})^{1/2}$
in the last inequality. If $\kappa>1$, then
\begin{align*}
|\xi|(1+|x+t\xi|^{2})^{\frac{\kappa-1}{2}} & \leq C|\xi|(1+|x|^{2}+|\xi|^{2})^{\frac{\kappa-1}{2}}\\
 & \leq C|\xi|(1+|\xi|^{2})^{\frac{\kappa-1}{2}}(1+|x|^{2})^{\frac{\kappa-1}{2}}\leq C(1+|\xi|^{2})^{\kappa/2}(1+|x|^{2})^{\frac{\kappa-1}{2}}.
\end{align*}
If $\kappa\in(0,1]$, then $|\xi|(1+|x+t\xi|^{2})^{\frac{\kappa-1}{2}}\leq|\xi|$.
In any case, we obtain, for $|\xi|>1$,
\begin{align*}
V_{1}(x+\xi)-V_{1}(x) & \leq\1_{(0,1]}(\kappa)C|\xi|+\1_{(1,\infty)}(\kappa)(1+|\xi|^{2})^{\kappa/2}(1+|x|^{2})^{\frac{\kappa-1}{2}}\\
 & \leq C\left(1+|\xi|+|\xi|^{\kappa}\right)(1+|x|^{2})^{\frac{\kappa-1}{2}}.
\end{align*}
Using $\left\langle \xi,\nabla V_{1}(x)\right\rangle \leq|\xi||\nabla V_{1}(x)|\leq C|\xi|(1+|x|^{2})^{\frac{\kappa-1}{2}}$
and

\[
V_{1}(x+\xi)-V_{1}(x)\leq V_{1}(x+\xi)\leq C(1+|x|^{2}+|\xi|^{2})^{\kappa/2}\leq CV_{1}(x)(1+|\xi|^{2})^{\kappa/2},
\]
for the integral against $\nu$, gives
\begin{align*}
 & \ \int\limits _{|\xi|>1}\left(V_{1}(x+\xi)-V_{1}(x)\right)\nu(d\xi)+\sum\limits _{i=1}^{m}x_{i}\int\limits _{|\xi|>1}\left(V_{1}(x+\xi)-V_{1}(x)-\left\langle \xi,\nabla V_{1}(x)\right\rangle \right)\mu_{i}(d\xi)\\
 & \leq CV_{1}(x)\int\limits _{|\xi|>1}(1+|\xi|^{2})^{\kappa/2}\nu(d\xi)+C(1+|x|^{2})^{\frac{\kappa-1}{2}}\sum\limits _{i=1}^{m}x_{i}\int\limits _{|\xi|>1}\left(1+|\xi|+|\xi|^{\kappa}\right)\mu_{i}(d\xi)\\
 & \leq CV_{1}(x)\left(\int\limits _{|\xi|>1}\left(1+|\xi|^{\kappa}\right)\nu(d\xi)+\sum\limits _{i=1}^{m}\int\limits _{|\xi|>1}\left(1+|\xi|+|\xi|^{\kappa}\right)\mu_{i}(d\xi)\right),
\end{align*}
where we have used $x_{i}\leq|x|\leq(1+|x|^{2})^{1/2}$, $i\in\{1,\dots,m\}$.
It remains to estimate the corresponding integrals for $|\xi|\leq1$.
Applying twice the mean value theorem gives
\begin{align}
V_{1}(x+\xi)-V_{1}(x)-\left\langle \xi,\nabla V_{1}(x)\right\rangle  & =\int\limits _{0}^{1}\left\{ \left\langle \xi,\nabla V_{1}(x+t\xi)\right\rangle -\left\langle \xi,\nabla V_{1}(x)\right\rangle \right\} dt\nonumber \\
 & =\int\limits _{0}^{1}\int\limits _{0}^{t}\sum\limits _{k,l=1}^{d}\frac{\partial^{2}V_{1}(x+s\xi)}{\partial x_{k}\partial x_{l}}\xi_{k}\xi_{l}dsdt\nonumber \\
 & \leq C|\xi|^{2}\int\limits _{0}^{1}\int\limits _{0}^{t}(1+|x+s\xi|^{2})^{\frac{\kappa-2}{2}}dsdt,\label{eq: double mean value theorem}
\end{align}
where we have used $\xi_{k}\xi_{l}\leq\frac{\xi_{k}^{2}+\xi_{l}^{2}}{2}\leq|\xi|^{2}$.
Using, for $i\in I$ and $|\xi|\le1$,
\begin{align*}
(1+x_{i})(1+|x+s\xi|^{2})^{\frac{\kappa-2}{2}} & \leq(1+|y+s\xi_{I}|^{2})^{1/2}(1+|x+s\xi|^{2})^{\frac{\kappa-2}{2}}\\
 & \leq(1+|x+s\xi|^{2})^{\frac{\kappa-1}{2}}\\
 & \leq(1+|x+s\xi|^{2})^{\kappa/2}\leq CV_{1}(x),
\end{align*}
we conclude that
\begin{align*}
 & \ \int\limits _{|\xi|\leq1}\left(V_{1}(x+\xi)-V_{1}(x)-\left\langle \xi,\nabla V_{1}(x)\right\rangle \right)\nu(d\xi)\\
 & \ \ \ +\sum\limits _{i=1}^{m}x_{i}\int\limits _{|\xi|\leq1}\left(V_{1}(x+\xi)-V_{1}(x)-\left\langle \xi,\nabla V_{1}(x)\right\rangle \right)\mu_{i}(d\xi)\\
 & \leq CV_{1}(x)\left(\int\limits _{|\xi|\leq1}|\xi|^{2}\nu(d\xi)+ \int\limits _{|\xi|\leq1}|\xi|^{2}\mu_{i}(d\xi)\right).
\end{align*}
Collecting all estimates proves the desired estimate for $\mathcal{A}_{1}$.
\end{proof} Let us now prove the desired estimate for $\mathcal{A}_{2}$.
\begin{Lemma} Suppose that the same conditions as in Proposition
\ref{PROP:MOMENT1} (b) are satisfied. Then there exists a constant
$C>0$ such that
\[
\mathcal{A}_{2}(x)\leq C\left(1+V_{2}(x)\right),\qquad x\in D.
\]
\end{Lemma} \begin{proof} Observe that $\nabla V_{2}(x)=\frac{2x}{1+|x|^{2}}$.
Hence we obtain for some generic constant $C>0$
\[
\left\langle \widetilde{b}+\beta x,\nabla V_{2}(x)\right\rangle \leq C\left(1+|x|\right)|\nabla V_{2}(x)|\leq C\frac{(1+|x|)|x|}{1+|x|^{2}}\leq C.
\]
Observe that, for $k,l\in\{1,\dots,d\}$,
\[
\frac{\partial^{2}V_{2}(x)}{\partial x_{k}\partial x_{l}} = \frac{2\delta_{kl}}{1+|x|^{2}}-\frac{4x_{k}x_{l}}{(1+|x|^{2})^{2}}.
\]
Using $x_{k}x_{l}\leq C(1+|x|^{2})$ gives $\left|\frac{\partial^{2}V_{2}(x)}{\partial x_{k}\partial x_{l}}\right|\leq\frac{C}{1+|x|^{2}}$.
This implies that
\begin{align*}
\sum\limits _{k,l=1}^{d}\left(a_{kl}+\sum\limits _{i=1}^{m}\alpha_{i,kl}x_{i}\right)\frac{\partial^{2}V_{2}(x)}{\partial x_{k}\partial x_{l}}\leq C\frac{1+|x|}{1+|x|^{2}}\leq C.
\end{align*}
Let us estimate the integrals against $\nu$ and $\mu_{1},\dots,\mu_{m}$.
Consider first the case $|\xi|>1$. Then
\[
V_{2}(x+\xi)-V_{2}(x)\leq V_{2}(x+\xi)\leq C\log(1+|x|^{2}+|\xi|^{2})\leq C\log(1+|x|^{2})+C\log(1+|\xi|^{2}),
\]
and hence we obtain
\begin{align*}
\int\limits _{|\xi|>1}\left(V_{2}(x+\xi)-V_{2}(x)\right)\nu(d\xi)\leq C\int\limits _{|\xi|>1}\left(V_{2}(x)+V_{2}(\xi)\right)\nu(d\xi)\leq C(1+V_{2}(x)).
\end{align*}
From the mean value theorem we obtain
\begin{align*}
V_{2}(x+\xi)-V_{2}(x)=\int\limits _{0}^{1}\left\langle \xi,\nabla V_{2}(x+t\xi)\right\rangle dt=2\int\limits _{0}^{1}\frac{\left\langle \xi,x+t\xi\right\rangle }{1+|x+t\xi|^{2}}dt\leq2|\xi|\int\limits _{0}^{1}\frac{|x+t\xi|}{1+|x+t\xi|^{2}}dt.
\end{align*}
In view of $x_{i}\leq x_{i}+t\xi_{i}\leq|x_{I}+t\xi_{I}|\leq|x+t\xi|$
for $i\in I$, we obtain $x_{i}(V_{2}(x+\xi)-V_{2}(x))\leq2|\xi|$.
Using $\left\langle \xi,\nabla V_{2}(x)\right\rangle \leq|\xi||\nabla V_{2}(x)|\leq C|\xi|$
gives
\begin{align*}
\sum\limits _{i=1}^{m}x_{i}\int\limits _{|\xi|>1}\left(V_{2}(x+\xi)-V_{2}(x)-\left\langle \xi,\nabla V_{2}(x)\right\rangle \right)\mu_{i}(d\xi)\leq C\sum\limits _{i=1}^{m}\int\limits _{|\xi|>1}|\xi|\mu_{i}(d\xi).
\end{align*}
It remains to estimate the corresponding integrals for $|\xi|\leq1$.
As in \eqref{eq: double mean value theorem}, we get
\begin{align*}
V_{2}(x+\xi)-V_{2}(x)-\left\langle \xi,\nabla V_{2}(x)\right\rangle  & \leq C|\xi|^{2}\int\limits _{0}^{1}\int\limits _{0}^{t}\frac{1}{1+|x+s\xi|^{2}}dsdt.
\end{align*}
This implies
\[
\int\limits _{|\xi|\leq1}\left(V_{2}(x+\xi)-V_{2}(x)-\left\langle \xi,\nabla V_{2}(x)\right\rangle \right)\nu(d\xi)\leq C\int\limits _{|\xi|\leq1}|\xi|^{2}\nu(d\xi).
\]
For $i\in I$, by $x_{i}\leq|x+s\xi|$, we get $\frac{x_{i}}{1+|x+s\xi|^{2}}\leq1$
and hence
\[
\sum\limits _{i=1}^{m}x_{i}\int\limits _{|\xi|\leq1}\left(V_{2}(x+\xi)-V_{2}(x)-\left\langle \xi,\nabla V_{2}(x)\right\rangle \right)\mu_{i}(d\xi)\leq C\sum\limits _{i=1}^{m}\int\limits _{|\xi|\leq1}|\xi|^{2}\mu_{i}(d\xi).
\]
Collecting all estimates proves the desired estimate for $\mathcal{A}_{2}$.
\end{proof}

\subsection{Some estimate on the Wasserstein distance}

Here and below we let $d\in\{d_{\kappa},d_{\log}\}$. Below we provide
two simple and known estimates for Wasserstein distances. \begin{Lemma}\label{WASSERSTEIN}
Let $f,\widetilde{f},g\in\mathcal{P}_{d}(D)$. Then
\[
W_{d}(f\ast g,\widetilde{f}\ast g)\leq W_{d}(f,\widetilde{f}).
\]
\end{Lemma} \begin{proof} Using the Kantorovich duality (see \citep[Theorem 5.10, Case 5.16]{V09},
we obtain
\begin{align*}
W_{d}(f\ast g,\widetilde{f}\ast g)=\sup\limits _{\|h\|\leq1}\left(\int\limits _{D}h(x)(f\ast g)(dx)-\int\limits _{D}h(x)(\widetilde{f}\ast g)(dx)\right),
\end{align*}
where $\|h\|=\sup_{x\neq x'}\frac{|h(x)-h(x')|}{d(x,x')}$. Using
now the definition of the convolution on the right-hand side gives
\begin{align*}
 & \ \int\limits _{D}h(x)(f\ast g)(dx)-\int\limits _{D}h(x)(\widetilde{f}\ast g)(dx)\\
 & =\int\limits _{D}\int\limits _{D}h(x+x')f(dx)g(dx')-\int\limits _{D}\int\limits _{D}h(x+x')\widetilde{f}(dx)g(dx')\\
 & =\int\limits _{D}\widetilde{h}(x)f(dx)-\int\limits _{D}\widetilde{h}(x)\widetilde{f}(dx),
\end{align*}
where $\widetilde{h}(x)=\int_{D}h(x+x')g(dx')$. Since $\|\widetilde{h}\|\leq1$,
we conclude that
\begin{align*}
W_{d}(f\ast g,\widetilde{f}\ast g) & =\sup\limits _{\|h\|\leq1}\left(\int\limits _{D}\widetilde{h}(x)f(dx)-\int\limits _{D}\widetilde{h}(x)\widetilde{f}(dx)\right)\\
 & \leq\sup\limits _{\|h\|\leq1}\left(\int\limits _{D}h(x)f(dx)-\int\limits _{D}h(x)\widetilde{f}(dx)\right)=W_{d}(f,\widetilde{f}),
\end{align*}
where we have used again the Kantorovich duality. This completes the
proof. \end{proof} The next estimate shows that the Wasserstein distance
is convex. For additional details we refer to \citep[Theorem 4.8]{V09}.
\begin{Lemma}\label{WASSERSTEIN:1} Let $P(x,\cdot)$ be a Markov
transition function on $D\times\mathcal{P}_{d}(D)$. Then, for any
$f,g\in\mathcal{P}_{d}(D)$ and any coupling $H$ of $(f,g)$, it
holds that
\[
W_{d}\left(\int\limits _{D}P(x,\cdot)f(dx),\int\limits _{D}P(x,\cdot)g(dx)\right)\leq\int\limits _{D\times D}W_{d}(P(x,\cdot),P(\widetilde{x},\cdot))H(dx,d\widetilde{x}).
\]
\end{Lemma}

\subsection{Proof of the elementary inequality with respect to $\log$}
Below we prove the following inequality.
\begin{Lemma}
 For any $a,b \geq 0$ one has
 \[
  \log(1 + ab) \leq \log(2e-1) \min\{\log(1+a), \log(1+b)\} + \log(2e-1) \log(1+a)\log(1+b).
 \]
\end{Lemma}
\begin{proof}
 Using the elementary inequality
 $\log(e + ab) \leq \log(e+a)\log(e+b)$, see \cite{MR1041496}, we easily obtain
 \begin{align*}
  \log(1 + ab) &= \log(e^{-1}) + \log(e + eab)
  \\ &\leq \log(e+a)\left( \log(e^{-1}) + \log(e + eb)\right)
  \leq \log(e+a)\log(1+b)
 \end{align*}
 from which we readily deduce
 \[
  \log(1+ab) \leq \min\{ \log(e+a)\log(1+b), \log(e+b)\log(1+a)\}.
 \]
 Fix any $\e > 0$. If $a \geq \e$, then we obtain
 \[
  \log(1 + ab) \leq \log(e+a)\log(1+b)
  \leq \frac{\log(e + \e)}{\log(1 + \e)} \log(1 + a)\log(1+b).
 \]
 The case $b \geq \e$ can be treated in the same way.
 Finally, if $0 \leq a,b \leq \e$, then we obtain
  \begin{align*}
  \log(1 + ab) &\leq \min\{ \log(e+a)\log(1+b), \log(e+b)\log(1+a)\}
  \\ &\leq \log(e + \e)\min \left\{ \log(e + \e), \frac{\log(e+\e)}{\log(1+\e)} \right\}.
 \end{align*}
 Collecting both estimates gives, for all $a,b \geq 0$, the estimate
 \[
  \log(1+ab) \leq g(\e)\min\{\log(1+a), \log(1+b)\} + g(\e)\log(1+a)\log(1+b),
 \]
 where $g(\e) = \min \left\{ \log(e + \e), \frac{\log(e+\e)}{\log(1+\e)} \right\}$.
 A simple extreme value analysis shows that $g$ attains its maximum at $\e = e-1$ which gives
 $\inf \limits_{\e > 0}g(\e) = g(e - 1) = \log(2e-1)$.
\end{proof}

\subsection*{Acknowledgments}

The authors would like to thank Jonas Kremer for several discussions
on affine processes and pointing out some interesting references on
this topic. Peng Jin is supported by the STU Scientific
Research Foundation for Talents (No. NTF18023).

\begin{footnotesize}

\end{footnotesize}
\end{document}